\documentclass[a4paper, 12pt, oneside]{amsart}
\usepackage[T1]{fontenc}
\usepackage[utf8]{inputenc}
\usepackage[dvipsnames]{xcolor}
\usepackage{amsbsy,amsmath,amsthm,amssymb,hyperref,paralist}
\usepackage[draft,multiuser,layout={inline,index}]{fixme}
\usepackage[edge-length=1cm]{dynkin-diagrams}
\usepackage{xfrac}
\usepackage{faktor}
\usepackage{tikz-cd}
\usepackage{tikz}
\usepackage{subcaption}
\usepackage{accents}
\usepackage{paralist}
\usepackage{scalerel}
\usepackage{stackengine}
\usepackage{placeins}
\usepackage{mathtools}
\usepackage{booktabs}
\usepackage{microtype}


\usetikzlibrary{positioning}

\fxsetup{theme=color}
\FXRegisterAuthor{a}{anp}{AS}
\FXRegisterAuthor{h}{ans}{HS}

\usepackage[
hyperref=true,backref=true,
backend=bibtex,
style=alphabetic, % draft
citestyle=alphabetic, % draft
sorting=nyt,
url = false,
doi = false,
isbn = false,
giveninits=true
]{biblatex}  % bibliográfia
\usepackage{csquotes}
\newtheorem{maintheorem}{Theorem}
\newtheorem{theorem}{Theorem}[section]
\newtheorem{lemma}[theorem]{Lemma}
\newtheorem{proposition}[theorem]{Proposition}

\newtheorem{corollary}[theorem]{Corollary}
\newtheorem{remark}[theorem]{Remark}

\theoremstyle{definition}

\newtheorem{definition}[theorem]{Definition}
\newtheorem{example}[theorem]{Example}

\newcommand{\CC}{\mathbb{C}}

\newcommand{\RR}{\mathbb{R}}

\newcommand{\QQ}{\mathbb{Q}}
\newcommand{\ZZ}{\mathbb{Z}}

\newcommand{\A}{\mathbb{A}}

\DeclareMathOperator{\spec}{Spec}

\DeclareMathOperator{\supp}{supp}

\DeclareMathOperator{\relint}{relint}

\DeclareMathOperator{\Hom}{Hom}

\DeclareMathOperator{\Der}{Der}

\DeclareMathOperator{\Rel}{L}

\title[First-order deformations of seminormal toric varieties]{On the first-order deformations of seminormal affine toric varieties}
\author{Konstantin Egert and Hendrik Sü{\ss}}

\begin{document}
\maketitle

\begin{abstract}
  We provide a general description of the module $T^1_X$ of first-order infinitesimal deformations of a not necessarily normal affine toric variety $X$ with a special emphasis on  seminormal varieties satisfying Serre's $(S_2)$ condition. In the surface case we are able to give a very detailed and concrete picture which includes the dimension of each homogeneous component.
\end{abstract}

\section{Introduction}
The purpose of this paper is to take first steps towards an understanding of the deformation theory of seminormal toric varieties. An affine variety is called seminormal if every rational function $f$, for which $f^2$ and $f^3$ are regular,  is itself regular. There is also an obvious notion of seminormalisation. Over the complex numbers the seminormalisation is known to be a homeomorphism on the level of Euclidean topology. Hence, from a topological point of view it is sufficient to study seminormal singularities. Similarly when studying F-split singularities in positive characteristic it is sufficient to consider seminormal ones by \cite{zbMATH03542542}. Seminormal affine toric varieties have been studied e.g. in  \cite{zbMATH01584951,zbMATH05072890,zbMATH07767754, arXiv:2501.12091}. On the other hand, deformations of non-normal toric surfaces singularities and their Milnor fibres have  been considered in the context of fillings for contact structures on lens spaces  e.g. in \cite{zbMATH05183764,zbMATH05183778,sandorthesis}.

For now, our aim is to provide a description of the vector space $T^1_X$ of first order deformations of a seminormal normal affine toric variety $X$. In the normal case this has been done by Altmann in \cite{zbMATH00681806}. For a normal affine toric variety $X=X_\sigma$ associated to a polyhedral cone $\sigma$ the vector space $T^1_X$ is graded by the character lattice of the torus and the homogeneous components  can be expressed in terms of linear relations between elements of the Hilbert basis of the weight monoid consisting of the lattice points inside $\sigma^\vee$.
Altmann's approach uses normality at a crucial point and seems to fail in the non-normal setting. Here, we use a more direct approach, but arrive at a general formula very much in the same spirit as the one given in \cite{zbMATH00681806}.

Assume for now, that we work over a field of characteristic $0$. For a subset $E=\{w_0,\ldots,w_{r+1}\}$  of the character lattice we consider relations $q \colon \sum_{i} q^+_i w_i = \sum_i q^-_iw_i$ between the generators with coefficients $q_i^\pm \in \ZZ_{\geq 0}$. Those form a finitely generated  abelian group, which we denote by $\Rel(E)$. We set $\bar q = \sum_{i} q^+_i w_i$ ($=\sum_i q^-_iw_i$).
\begin{maintheorem}[see {Theorem~\ref{thm:t1-general-formula}}]
  Assume $X=\spec k[S]$ is a not necessarily normal affine toric variety and $E\subset S$ is a generating set of the corresponding weight monoid $S$. Then the homogeneous component of first-order deformations of degree $u$ is given by
   \[(T^1_X)_u \cong \left(\Rel(E^u_S) /\langle q \mid \bar q + u \notin S \rangle \right)^*\otimes k\]
   where $E^u_S=\{w \in E \mid u+w \notin S\}$.
 \end{maintheorem}
 
If we additionally assume  Serre's $(S_2)$ property and  seminormality, our formula simplifies. In this form it allows us to completely settle the case of a seminormal toric surface singularity $X$. Here, the normalisation $\tilde X$ is given by a two-dimensional cone $\sigma \subset \RR^2$ spanned by primitive lattice elements $v_1, v_2 \in \sigma \cap \ZZ^2$ and its dual cone $\sigma^\vee$. The restriction of the normalisation map to each of the two torus invariant curves is a covering of index $\ell_1$ and $\ell_2$ respectively. Let us number the lattice points on the
line $\{ u  \mid \langle v_i, u \rangle=j\}$ consecutively, where we choose the $0$-point  in such a way that the lattice points with positive numbers are exactly the ones which pair positively with the other ray generator. Let us write $u_{i}^{j}(m)$ for the $m$th lattice point within this numbering. We then have a  very explicit description of the dimensions of the graded components $(T^1_X)_u$ in terms of the dimensions of the corresponding graded components $(T^1_{\tilde X})_u$, of the space of first-order deformations of the normalisation.
\begin{maintheorem}[see Theorem~\ref{thm:T1dim-surface}] Let $X$ be a seminormal affine toric surface
  with covering indices $\ell_1,\ell_2 \geq 3$ over the non-normal locus and normalisation $\tilde X$. Then we have 
      \begin{equation*}
        \dim(T^1_{X})_u= d_1(u)+d_2(u)+\dim (T_{\tilde X}^1)_u
      \end{equation*}
      with
      \[d_i(u)=
        \begin{cases}
          1 & u=-u_{i}^{0}(m),\, 1\leq m < \ell_i\\
          \ell_i-2 & u=-u_{i}^{1}(m),\, m \leq 0\\
          \ell_i-1 & u=-u_{i}^{1}(1)\\
          \ell_i-m + 1& u=-u_{i}^{1}(m),\, 2\leq m < \ell_i\\
          1 & u=-u_{i}^{1}(\ell_i),\; \tilde X \not\cong \mathbb{A}^2\\
          0 & \text{otherwise}
        \end{cases}
      \]
      for $i=1,2$.           
\end{maintheorem}

The paper is organised as follows. In §\ref{sec:seminormal-toric} we recall the description of seminormal affine toric varieties and the characterisation of Serre's $(S_2)$-property. In §\ref{sec:general-toric-t1} we prove our $T^1$-formula for general not necessarily normal affine toric varieties and discuss simplifications under the additional assumption of the $(S_2)$-property. In §\ref{sec:seminormal-case} we prove some preliminary results for seminormal affine toric varieties, which are then deployed in §\ref{sec:surfaces} to settle the surface case in characteristic $0$.

\subsection*{Acknowledgement} We thank Milena Hering for helpful discussions and in particular for pointing us to Schreyer's Theorem.
\section{Seminormal toric varieties}
\label{sec:seminormal-toric}
Let $M$ and $N$ be dual lattices and $M_\QQ=M \otimes \QQ$, $N_\QQ=N \otimes \QQ$ the corresponding $\QQ$-vector spaces. The pairing between dual lattices and vector spaces will be denoted by $\langle\cdot ,\cdot \rangle$. For a polyhedral cone $\sigma \subset N_\QQ$ we consider the dual cone $\sigma^\vee =  \{u \in M_\QQ \mid \forall_{v \in \sigma} \colon \langle u, v \rangle \geq 0\}$. We use the notation $\tau \prec \sigma$ to denote the face relation between cones $\tau$ and $\sigma$. The dual face $\tau^* \prec \sigma^\vee$ is defined by $\tau^\perp \cap \sigma^\vee$. For a ray $\rho$ its primitive lattice generator is denoted by $v_\rho$.

From now on we consider affine monoids $S \subset M$ spanning $M$ as an abelian group. We set $\sigma=(\QQ_{\geq 0}S)^\vee$. For general, not necessarily normal, affine toric varieties $X=\spec k[S]$ we have a correspondence between invariant open subsets and faces of $\sigma$. Indeed, for a face $\tau \prec \sigma$ the monoid ring  $k[S_\tau]$ with $S_\tau:=S + \ZZ(\tau^\perp \cap S)$ is a localisation of $k[S]$, defining an open invariant subset $U_\tau$ of $X$. Note, that by construction the dual cone of $\QQ_{\geq 0}S_\tau$ is $\tau$.

\begin{definition}
  We say that $S$ \emph{satisfies $(S_2)$} if $S=\bigcap_{\rho \in \sigma(1)} S_\rho$ holds.
\end{definition}
\begin{definition}
  In our setup the monoid $S$ is called \emph{normal} if $S=\bar S:= \sigma^\vee \cap M$. In general $\bar S$ is called the \emph{normalisation} of  $S$.
  
  The monoid $S$ is called \emph{seminormal}
  if
  \begin{equation}
    \label{eq:seminormality}
    S= \bigcup_{\tau \prec \sigma} G_\tau \cap \relint(\sigma^\vee  \cap \tau^\perp)
  \end{equation}
  holds, with $G_\tau=\langle S \cap \tau^\perp \rangle$ being the abelian subgroup generated by $S \cap \tau^\perp$.
  
  We call  $G_\tau \subset  \tau^\perp \cap M$ the \emph{sublattice associated to} $\tau \prec \sigma$ via $S$ and
  \begin{equation}
  G_\tau^*/(\tau^\perp \cap M)^* %(\tau^\perp \cap M)/G_\tau 
  \label{eq:assoc-finite-quot}  
\end{equation}
  the \emph{Galois group associated to} $\tau \prec \sigma$.
\end{definition}

\begin{remark}
  Note, that a seminormal monoid satisfies $(S_2)$ if and only if for every face $\tau \prec \sigma$
  \[
   G_\tau = \bigcap_{\tau \succ \rho \in \sigma(1)} G_\rho
  \]
  holds.
\end{remark}

By \cite{zbMATH03542542} in conjunction with \cite{zbMATH01584951} the monoid ring $k[S]$ is seminormal if and only if $S$ is seminormal in the above sense. By \cite[Cor.~3.4]{zbMATH05072890} the monoid ring satisfies $(S_2)$ if and only if the monoid satisfies $(S_2)$ in our sense.
\begin{remark}
  If $X=\spec \CC[S]$ is semi-normal, then the non-normal locus of $X$ consists exactly of the orbit closures
  which correspond to faces with non-trivial associated Galois group. Moreover, this group is, indeed, the Galois group of the covering obtained by restricting the normalisation map to the corresponding orbit.
\end{remark}

\begin{remark}
  If a monoid $S$ is normal it is completely determined by the cone $\sigma={(\QQ_{\geq 0}S)}^\vee$. If it is seminormal we additionally need the sublattices associated to the faces of $\sigma$ to be able to reconstruct $S$. Moreover, if $S$ also satisfies $(S_2)$, then it is already determined by $\sigma$ together with the sublattices associated to the rays.
\end{remark}

\begin{remark}
   It is straight forward to check, that in the case of a seminormal monoid $S$, the monoid $S_\tau$ is again seminormal and the sublattices associated to the faces of $\tau$ via $S_\tau$ are the same as those associated them as faces of $\sigma$ via $S$.
\end{remark}
\section{A general toric  $T^1$-formula}
\label{sec:general-toric-t1}
We study the graded components $(T^1_X)_u$ of the space of first order infinitesimal deformations of not necessarily normal affine toric varieties. We mainly follow the notation and ideas of \cite{zbMATH00681806} and \cite{zbMATH00895853}.

Consider a $k$-algebra $A=R/I$ with $R=k[z_1,\ldots,z_r]$ and the corresponding affine variety $X=\spec A$. The $A$-module $T^1_X$ parametrises first order deformations of $A$.  It can be calculated as
\begin{equation}
  T^1_X={\Hom_A(I/I^2,A)}/\Der_A(I/I^2,A),\label{eq:t1-general-formula}
\end{equation}
where $\Der_A(I/I^2,A)$ is the submodule generated by the homomorphisms $\partial_i \colon  f \mapsto \partial f/\partial z_i$. This definition is up to isomorphism independent of the representation $A=R/I$. If the $A$ is $M$-graded, all involved modules are naturally $M$-graded as well. Let us write e.g. $\Hom_A(I/I^2,A)_u$ for the vector space of homogeneous elements of degree $u$.

We consider the following general setup. Let $E=\{w_1,\ldots,w_r\} \subset S$ be a finite subset of an affine monoid $S$ and let $M$ be the free abelian group generated by $S$. This gives rise to  homomorphisms
\[\pi \colon \ZZ^r \to M,\; e_i \mapsto w_i, \qquad \pi|_{\ZZ_{\geq 0}^r} \colon \ZZ_{\geq 0}^r \to S.\]
and
We call \(\Rel(E)=\ker  \pi\) the group of integral relations of the set $E$. We write $\supp q= \{w_i \mid q_i \neq 0\}$ for the \emph{support} of a relation. Every relation $q$ can be written uniquely as $q^+-q^-$ with $q^+,q^- \in \ZZ_{\geq 0}^r$, such that the supports of $q^+$ and $q^-$ are disjoint. We sometimes write $\bar q$ as a short form for $\pi(q^+)$ (which coincides with $\pi(q^-)$). In the following if we write
$a-b \in \Rel(E)$, then it should be understood, that $a,b \in \ZZ_{\geq 0}^r$. Sometimes it will be more convenient to denote the relation $q=(q_1,\ldots,q_r)$ by $q \colon \sum_{i=1}^r q_r \cdot w_r=0$.

We also obtain an induced $k$-algebra homomorphism
\[k[z_1,\ldots,z_r] \to k[S],\; z_i \mapsto \chi^{w_i}.\]
The kernel $I$ of this homomorphism is generated by elements of the form $z^a-z^b$ with $a-b \in \Rel(E)$.
If $E$ was a generating set, then all the constructed homomorphisms are surjective and 
$k[S]$ is isomorphic to $k[z_1,\ldots,z_r]/I$. We set $\Rel_k(E) = \ker (\pi \otimes k)$. In general, we then have $\Rel(E) \otimes k \subset \Rel_k(E)$. In positive characteristic this inclusion can be strict. In any case, for subsets $G \subset \Rel(E)$ we may consider its $k$-span $\langle G \rangle_k:=\langle G \rangle \otimes k \subset \Rel_k(E)$. In the following we assume $E\subset S$ to be a generating set.
\begin{lemma}
  \label{lem:hom-iso}
  For $X=\spec k[S]$ and $E\subset S$ a generating set the map
  \begin{align}
    \Hom(\Rel(E) /\langle q \mid \bar q + u \notin S \rangle,k) &\to \Hom_{k[S]}(I/I^2,k[S])_u \label{eq:vs-iso}\\
    \varphi \quad&\mapsto \; \left(z^a - z^b \mapsto \varphi(a-b) \chi^{\pi(a)+u}\right) \nonumber
  \end{align}
  defines an isomorphism of vector spaces.

  Moreover, this extends to an isomorphism of $k[S]$-modules
  \begin{equation}
\bigoplus_{u \in M}\Hom(\Rel(E) /\langle q \mid \bar q + u \notin S \rangle,k) \stackrel{\sim}{\longrightarrow} \Hom_{k[S]}(I/I^2,k[S])\label{eq:module-iso}
\end{equation}
  where for $w \in S$ the multiplication with $\chi^w$ on the left-hand-side is given via the natural inclusion
  \[ \Hom(\Rel(E) /\langle q \mid \bar q + u \notin S \rangle,k) \hookrightarrow \Hom(\Rel(E) /\langle q \mid \bar q + u +w \notin S \rangle,k).\]
\end{lemma}
\begin{proof}
  The ideal $I$ is naturally $M$-graded via $\deg(z^a-z^b)=\pi(a)=\pi(b)$. A homogeneous element  $\Phi \in \Hom(I/I^2, k[S])$ of degree $u$ sends $z^a-z^b$ therefore to some element  in $k \cdot \chi^{\pi(a)+u}$. Let us denote the coefficient in front of $\chi^{\pi(a)+u}$ by $\varphi(a-b)$.  This defines a map $\varphi \colon \Rel(E) \to k$. Indeed, assume we have $a'-b'=a-b$. Then
  $a'=a+c$ and $b'=b+c$ and
  \[\Phi(z^{a'}-z^{b'})=\Phi(z^c(z^a-z^b))=\Phi(z^a-z^b)\chi^{\pi(c)}=\varphi(a-b)\chi^{\pi(a')+u}.\]
  Since $\varphi(a-b)\chi^{\pi(a)+u} \in k[S]$ holds, the map has to vanish on the relations $\langle q \in \Rel(E) \mid \bar q + u \notin S \rangle$.  It remains to show that the map $\varphi$  is a homomorphism. To this end we observe
  \begin{align*}
    \Phi(z^{a+a'}-z^{b+b'}) &=  \Phi(z^{a'}(z^{a}-z^{b})+z^{b}(z^{a'}-z^{b'}))\\
                            &= \Phi(z^{a'}(z^{a}-z^{b})) + \Phi(z^{b}(z^{a'}-z^{b'}))\\
                            &=\varphi(a-b)\chi^{\pi(a+a')+u} + \varphi(a'-b')\chi^{\pi(a'+b)+u}\\
                            &=\varphi(a-b)\chi^{\pi(a+a')+u} + \varphi(a'-b')\chi^{\pi(a+a')+u}\\
                            &=(\varphi(a-b) + \varphi(a'-b'))\chi^{\pi(a+a')+u}.
  \end{align*}
  Hence, $\varphi((a+a)'-(b+b'))=\varphi(a-b)+ \varphi(a'-b')$ holds.

  On the other hand, every group homomorphism
  \[\varphi \colon \Rel(E) /\langle q \mid \bar q + u \notin S \rangle \to k\]
  induces a module homomorphism $\Phi \colon I/I^2 \to k[S]$ of degree $u$ which can be defined on the elements of a binomial Gröbner basis of $I$ by   $\Phi(z^a-z^b)= \varphi(a-b)\chi^{\pi(a)+u}$. That $\Phi$ is indeed well defined in this way, follows from Lemma~\ref{lem:syzygies}, below.  It is also easy to see that both assignments are $k$-linear and inverse to each other.

  The last claim just follows by tracing the multiplication by $\chi^w$ through the constructed isomorphism.
\end{proof}

\begin{lemma}
  \label{lem:syzygies}
In the situation described above, assume \[G=\{(z^{a_1}-z^{b_1}),\ldots, (z^{a_m}-z^{b_m})\}\] is a Gröbner basis for $I$.
Then the syzygies between the corresponding generators of $I/I^2$ are generated by relations of the form
  \begin{equation}
    \label{eq:syzygy}
    \sum_i \lambda_i \chi^{u_i} \cdot (z^{a_i}-z^{b_i})=0 \text{ with }
      \sum_i \lambda_i (a_i-b_i)=0
    \end{equation}
\end{lemma}
\begin{proof}
  We make the following elementary, yet, crucial observation. Assume we have $z^a-z^b, z^{a'}-z^{b'} \in I$
  and choose $c,c' \in \ZZ_{\geq 0}^r$ such that
  \begin{equation}
    \label{eq:s-poly-and-reduction}
    z^c(z^a-z^b)\pm z^{c'} (z^{a'}-z^{b'}) = z^{a''}-z^{b''}
  \end{equation}
  is a binomial again. Then $(a-b) \pm (a'-b') - (a''-b'')=0$ holds. Hence, every relation 
  \begin{equation}
    \label{eq:induction-step}
     \chi^{\pi(c)}(z^a-z^b)\pm \chi^{\pi(c')} (z^{a'}-z^{b'}) - (z^{a''}-z^{b''})=0,
  \end{equation}
  over $k[S]$ arising in this way, fulfils indeed the condition in (\ref{eq:syzygy}). On the other hand, $k$-linear combinations of relations of the type (\ref{eq:syzygy}) will be again of this type.

   Forming Buchberger's S-polynomial for a pair of elements of the Gröbner basis leads to a relation (\ref{eq:induction-step}), which involves two elements of the Gröbner basis and one new binomial $(z^{a''}-z^{b''})$.
   We now fully reduce $(z^{a''}-z^{b''})$ with respect to $G$, but each reduction step corresponds to an identity as in (\ref{eq:s-poly-and-reduction}).   For each such reduction we add the corresponding relation of the type (\ref{eq:induction-step}). By Bucherberger's criterion the full reduction with respect to $G$ gives $0$. Hence, we end up with a relation which only involves the images of the elements of $G$. Schreyer's Theorem \cite[Chapter 5]{zbMATH02163714} tells us that these relations indeed generate the full syzygy module.

  Note, that Schreyer's theorem actually deals with the syzygies between the elements of $G$. However, every syzygy between the images of the elements of $G$ in the $k[S]$-module $I/I^2$ can be lifted to a syzygy of the elements of $G$ as elements of the $k[z_1,\ldots,z_r]$-module $I$.
\end{proof}

For any subset $F \subset M$ and $u \in M$ we set $F_S^u = \{w \in F \mid u+w \notin S\}$.
\begin{theorem}
  \label{thm:t1-general-formula}
  Assume $X=\spec k[S]$ is a not necessarily normal toric variety and $E\subset S$ is a generating set of the monoid $S$. We have
   \[(T^1_X)_u \cong \left(\Rel_k(E^u_S) /\langle q \in \Rel(E^u_S) \mid \bar q + u \notin S \rangle_k \right)^*.\]

   Moreover, multiplication with $\chi^w$ on the left-hand-side corresponds to the natural linear map
   \[
     \left(\Rel_k(E^u_S) /\langle q  \mid \bar q + u \notin S \rangle_k \right)^* \to  \left(\Rel_k(E^{u+w}) /\langle q \mid \bar q + u+w \notin S \rangle_k \right)^*
   \]
   which is induced by the restriction of a linear form on $\Rel_k(E^u_S)$ to $\Rel_k(E^{u+w}) \subset \Rel_k(E^u_S)$. In particular, $\chi^w$ annihilates those elements in $(T^1_X)_u$, which correspond to linear forms that vanish on $\Rel_k(E^{u+w})$. 
 \end{theorem}
\begin{remark}
  Note that by \cite{zbMATH00895853} the torsion submodule of the module of differentials of $X$ is described by an almost identical formula, where the condition of \emph{not} being contained in $S$ just has to be replaced by the condition of being contained in $S$. 
\end{remark}
 \begin{proof}[Proof of Theorem~\ref{thm:t1-general-formula}]
   We have
   \begin{align*}
    \left(\Rel_k(E) /\langle q \mid \bar q + u \notin S \rangle_k \right)^*  &\cong \left(\Rel(E) /\langle q \mid \bar q + u \notin S \rangle \otimes k \right)^* \\
     &\cong \Hom(\Rel(E) /\langle q \mid \bar q + u \notin S \rangle,k)
   \end{align*}

   Now a homogeneous element of degree $u$ in $\Der_A(I/I^2,A)$ has the form
  \[\chi^{u+w_i} \partial_i \colon I/I^2 \to k[S],\quad f \mapsto \chi^{u+w_i} \cdot \partial f/\partial z_i\]
  with $u+w_i \in S$. The isomorphism, considered in the proof of Lemma~\ref{lem:hom-iso}, sends such an element to the projection
  \[p_i \colon \Rel(E) \to k, \; a-b \mapsto a_i-b_i,\]
  since \[\partial (z^a-z^b)/\partial z_i = (a_i-b_i)\chi^{\pi(a)-w_i}.\]
  Hence, we obtain
  \begin{align*}
    (T^1_X)_u &=  \left.\left(\Rel_k(E) /\langle q \mid \bar q + u \notin S \rangle_k \right)^*\middle/ \langle p_i \mid u+w_i \in S \rangle_k\right.\\
              &=\left(\left. (\Rel_k(E) \cap \langle p_i \mid  u+w_i \in S\rangle^\perp_k) \middle/ \langle q \mid \bar q + u \notin S \rangle_k \right. \right)^*\\
             &= \left(\Rel_k(E^u_S) /\langle q \mid \bar q + u \notin S \rangle_k \right)^*.
  \end{align*}
\end{proof}

\begin{lemma}
  \label{lem:involved-generators}
  Assume we have $E \subset S$ and $q \in L(E)$ with $\bar q+u \notin S$. In this situation we have $q \in L(E^u_S)$.
\end{lemma}
\begin{proof}
  Assume $w_i \in \supp q$. W.l.o.g we may assume that $w_i$ appears with positive coefficient in $q$. Then
  $q^+ -e_i \in \ZZ^r_{\geq}$ and we have
  \[\bar q+u=\pi(q^+-e_i)+\pi(e_i)+u=\pi(q^+-e_i)+w_i+u\]
  Since $\pi(q-e_i)\in S$ and $\bar q+u \notin S$ it follows that $w_i+u \notin S$. 
\end{proof}

\begin{lemma}
  \label{lem:S2-T1}
  Assume $X=\spec k[S]$ satisfies $(S_2)$, $\sigma=(\QQ_{\geq 0}S)^\vee$ and $E$ is a generating set of $S$. Then  we have
  \begin{equation}
    \sum_\rho \langle q \in \Rel(E^u_{S_\rho}) \mid \bar q+u \notin S_\rho\rangle  = \langle q \in \Rel(E^{u}) \mid \bar q + u \notin S \rangle \label{eq:S2-T1}
  \end{equation}
  and
  \begin{equation}
    \label{eq:S2-T1-formula}
     (T^1_X)_u \cong \left(\Rel_k(E^u_S) /\sum_\rho \langle q \in \Rel(E^u_{S_\rho}) \mid \bar q+u \notin S_\rho\rangle_k  \right)^*.
  \end{equation}
\end{lemma}
\begin{proof}
 Since we have $S \subset S_\rho$ for every $\rho \in \sigma(1)$, we get the inclusion
\begin{equation*}
\bigcup_\rho \{q \in \Rel(E^u_{S_\rho}) \mid \bar q+u \notin S_\rho\}  \subset  \left\{ q \in \Rel(E^{u}_S) \mid \bar q + u \notin S \right\}.\label{eq:bigcup}
\end{equation*}
Now, assume we have a relation $q \in \Rel(E^u_S)$ with $\bar q+u \notin S$. Then by the $(S_2)$-property we have $\bar q+u \notin S_\rho$ for at least one $\rho \in \sigma(1)$. By applying Lemma~\ref{lem:involved-generators} to $q \in \Rel(E^u_S)$ and $S_\rho$ we obtain $q \in \Rel(E^u_{S_\rho})$. This shows also the other inclusion and we obtain (\ref{eq:S2-T1}). The formula in (\ref{eq:S2-T1-formula}) follows by Theorem~\ref{thm:t1-general-formula}.
\end{proof}

\begin{lemma}
  \label{lem:quotients-under-localisation}
  Let $S$ be an affine monoid and $E \subset S$ an finite subset and $F = E \cup (-B) $ for some $B \subset E \cap S \cap (-S)$. In this situation we have
  \[
    \Rel_k(E^u_S)/ \langle q  \mid u + \bar q \notin S \rangle_k  \cong \Rel_k(F^u_S)/ \langle q \mid u + \bar q \notin S \rangle_k.
  \]
\end{lemma}
\begin{proof} By using  induction we can reduce the claim to the case that $B=\{w\}$ for some $w \in E \cap S \cap (-S)$.  We note, that either one has $u\in S$ -- in this case $F^u_S$ and $E^u_S$ are both empty -- or $F^u_S=E^u_S \cup \{-w\}$ holds. In the first case the claim holds trivially. Hence, we may assume that we are in the second case.   Now, we want to show the identity
  \begin{equation}
\langle q \in \Rel(E^u_S) \mid u + \bar q \notin S \rangle_k =  \langle q \in \Rel(F^u_S) \mid u + \bar q \notin S \rangle_k \cap \Rel_k(E^u_S).\label{eq:span-identity}
\end{equation}
  Indeed, one inclusion is clear. To show the other inclusion consider any relation from the right-hand-side $q=\sum_ i \lambda_i q_i  \in \Rel_k(E^u_S)$ with $q_i \in \Rel(F^u_S)$ and $\bar q_i + u  \notin S$.
  Assume $-w$ appears with coefficient $\mu_i$ in $q_i$. We by setting $q_i'=q_i - \mu_i q_w$, with $q_w\colon 1 \cdot w +  1 \cdot(-w)=0$, we obtain relations
  $q_i' \in \Rel(E^u_S)$. Since we have $\bar q'_i=\bar q_i+\mu_i w$ and $w+S=S$, it follows that we have the following equivalent conditions
  \[
    \bar q'_i +u  \notin S \quad\Leftrightarrow\quad \bar q_i +u  + \mu_i w  \notin S \quad\Leftrightarrow\quad  \bar q_i +u  \notin S.
  \]
Therefore $q_i'$ is even an element of $\langle q \in \Rel(E^u_S) \mid u + \bar q \notin S \rangle$. We can express $q$ as follows.
  \[
    q=\sum_i \lambda_i q_i' + \Big(\sum_i \lambda_i \mu_i \Big)q_{w}
  \]
  Since $q$ and $\sum_i \lambda_i q_i'$ are elements of $\Rel_k(E^u_S)$, but $q_w \colon 1 \cdot w + 1 \cdot (-w)=0$ is not, we must have $\left(\sum_i \lambda_i \mu_i \right)=0$ and $q$ is therefore an element of the left-hand-side of (\ref{eq:span-identity}).   Now, it follows from (\ref{eq:span-identity}) that we have, a natural inclusion
  \begin{equation}
    \Rel_k(E^u_S)/ \langle q  \mid u + \bar q \notin S \rangle_k  \;\hookrightarrow\; \Rel_k(F^u_S)/ \langle q \mid u + \bar q \notin S \rangle_k.\label{eq:embedding-of-quotients}
  \end{equation}
  Now, take any relation in $\Rel_k(F^u_S)$ by adding an appropriate multiple of $q_w$ we obtain a relation in $\Rel_k(E^u_S)$. But the relation $q_w$ is an element of $\langle q \mid u + \bar q \notin S \rangle$, since $\bar q_w=0$ and we assumed that $u \notin S$. Hence, every class in the right-hand-side of (\ref{eq:embedding-of-quotients}) has a representative in $\Rel_k(E^u_S)$. It follows that the inclusion in (\ref{eq:embedding-of-quotients}) is, in fact, an isomorphism.
\end{proof}
\begin{proposition}
  \label{prop:quotients-under-localisation}
  Consider a monoid $S$ with generating set $E$ and $X=\spec k[S]$. Then for $\tau \prec \sigma=(\QQ_{\geq 0}S)^\vee$ we have
  \[
   (T_{U_\tau}^1)_u \cong \left(\Rel_k(E^u_{S_\tau}) /\langle q \in \Rel(E^u_{S_\tau}) \mid \bar q + u \notin S_\tau \rangle_k \right)^*.
 \]
 although $E$ is not a generating set of $S_\tau$.
\end{proposition}
\begin{proof}
  For every $\tau \prec \sigma$ the set $E$ can be extended to a generating system $F=E \cup (-B)$ of $S_\tau$, where $B \subset E \cap \rho^\perp$.
We now apply Lemma~\ref{lem:quotients-under-localisation} to $E$ and $F$ and obtain
\begin{equation*}
  \Rel_k(E^{u}_{S_\tau})/\langle q \mid \bar q + u \notin S_\tau \rangle_k \cong \Rel_k(F^{u}_{S_\tau})/\langle q  \mid \bar q + u \notin S_\tau \rangle_k
  \cong (T_{U_\tau}^1)^*_u.
\end{equation*}
\end{proof}

\begin{definition}
  \label{def:critical}
   If $(T^1_{U_\rho})_u \neq 0$ holds for at  least one $\rho \in \sigma(1)$, then we say $u\in M$ is a \emph{critical weight} for $T^1_X$. Otherwise we say it is \emph{non-critical}.
\end{definition}

\begin{proposition}
  \label{prop:S2-formula}
  Assume $X=\spec k[S]$ satisfies $(S_2)$, $\sigma=(\QQ_{\geq 0}S)^\vee$ and $E$ is a generating set of $S$. For every non-critical weight $u\in M$  we have
  \begin{equation}
    \label{eq:normal-formula}
    (T^1_X)_u=\left( \Rel_k(E^u_S) \middle/ \sum_{\rho\in \sigma(1)} \Rel_k(E^u_{S_\rho}) \right)^*
  \end{equation}
  with $E^u_{S_\rho}= \{w \in E \mid w +u \notin S_\rho \}$.
\end{proposition}

\begin{proof}
By Proposition~\ref{prop:quotients-under-localisation} we have  $\left(\Rel_k(E^{u}_{S_\tau})/\langle q \mid \bar q + u \notin S_\tau \rangle_k\right)^* 
 \cong (T_{U_\tau}^1)_u$. But  we assumed $T^1_{U_\rho}=0$. Hence we get the identity \[\Rel_k(E^u_{S_\rho})=\langle q \in \Rel(E^u_{S_\rho}) \mid \bar q + u \notin S_\rho \rangle_k.\]
Now, Lemma~\ref{lem:S2-T1}  implies the claim.
\end{proof}

We recover a result by Altmann from \cite{zbMATH00681806}.
\begin{corollary}
  \label{cor:normal-case}
  Assume $X=\spec k[S]$ is normal, $\sigma=(\QQ_{\geq 0}S)^\vee$ and $E$ is a generating set of $S$. Then we have
  \begin{equation*}
    (T^1_X)_u=\left( \Rel_k(E^u) \middle/ \sum_{\rho\in \sigma(1)} \Rel_k(E^u_{\rho}) \right)^*
  \end{equation*}
  with  $E^u=\{w \in E \mid w +u \notin  \sigma^\vee\}$ and $E^u_\rho= \{w \in E \mid w +u \notin  \rho^\vee\}$.
\end{corollary}
\begin{proof}
  Since normality is equivalent to the $(S_2)$-property and regularity in codimension $1$, the assumptions of Proposition~\ref{prop:S2-formula} are automatically fulfilled for every degree $u \in M$. Hence, formula (\ref{eq:normal-formula}) holds. On the other hand, we have $\bar S=S$ and therefore
  $E^u=E^u_S$ and $E^u_\rho=E^u_{S_\rho}$.
\end{proof}
\section{The seminormal case}
\label{sec:seminormal-case}
Our the calculation of $T^1$ according to  formula (\ref{eq:t1-general-formula}) commutes with localisation. In particular, for a seminormal toric variety $X=\spec k[S]$ we have the localisation map
 \begin{equation}
  T^1_X \to T^1_{U_\tau}=(T^1_X)_{\chi^w}\label{eq:localisation}
\end{equation}
for $w \in \sigma^\vee \cap S$, $\tau=w^\perp \cap \sigma$  and
\[U_\tau=\spec k[S+G_\tau] = \spec k[S_w]=\spec k[\underbrace{S-\ZZ_{\geq 0} w}_{=:S_w}]=\spec k[S]_{\chi^w}.\]

  \begin{lemma}
    \label{lem:localisation}
    Assume $X=\spec k[S]$ is an affine  seminormal toric variety corresponding to the monoid $S$ and the cone $\sigma = (\QQ_{\geq 0}S)^\vee$.
    Then the  localisation map (\ref{eq:localisation}) is bijective in those degrees
    $u$, which have the property that $\langle u, v_\rho \rangle > 0$ for every $\rho \in \sigma(1)\setminus \tau(1)$.
  \end{lemma}
  \begin{proof}
    Take any $w' \in S$. We claim that we have $w'+u \in S$ if and only if  $w+w'+u \in S$ holds. One direction is trivial. For the other one assume $w+w'+u \in S$ and set $(w+w'+u)^\perp \cap \sigma=:\delta \prec \tau$. Now, we make two observations
    \begin{enumerate}
    \item We see that $w'+u \in \delta^\perp$, since we have $w+w'+u \in \delta^\perp$ and $-w \in \tau^\perp \subset \delta^\perp$.
    
    \item For every $\rho \in \sigma(1)\setminus \delta(1)$ we have $w'+u \in \relint \rho^\vee$. Indeed, for $\rho \in \sigma(1)\setminus \tau(1)$ we have $w' \in \rho^\vee$ (since $w' \in \sigma^\vee$) and $u \in \relint \rho^\vee$ (by our condition on $u$). For $\rho \in \tau(1) \setminus \delta(1)$ we have $w+w'+u \in \relint \rho^\vee$ and $-w \in \rho^\vee$ (since $w \in \tau^\perp$). Hence, we conclude that $w'+u \in \relint \rho^\vee$. 
    \end{enumerate}
    From both observations together it follows that $w'+u \in \relint (\sigma^\vee \cap \delta^\perp)$. 
    On the other hand, we also have $w \in G_\tau \subset G_\delta$ and $w+w'+u \in \delta^\perp \cap S \subset G_\delta$. Hence, we obtain $w'+u \in G_\delta \cap \relint (\sigma^\vee \cap \delta^\perp) \subset S$ by our condition (\ref{eq:seminormality}) for seminormality.

    The equivalence $w'+u \in S \Leftrightarrow w+w'+u \in S$ implies that for any subset $E \subset S$ we have  $E^u_S=E^{u+w}_S$ and \[\langle q  \in \Rel_k(E^u_S)\mid \bar q + u \notin S \rangle = \langle q  \in \Rel_k(E^{u+w}_S)\mid \bar q + u+w \notin S \rangle.\] 
    Hence, by Theorem~\ref{thm:t1-general-formula} the linear map $\chi^w  \colon (T^1_X)_u  \to (T^1_X)_{u+w}$ is an isomorphism. It follows, that also the localisation map is an isomorphism in these degrees.
  \end{proof}

  The following proposition determines $T^1_X$ in those degrees which are non-positive along at most one ray of $\sigma=(\QQ_{\geq 0}S)^\vee$.
  \begin{proposition}
  \label{prop:t1}
  Assume $X=\spec k[S]$ is an affine  seminormal toric variety corresponding to the monoid $S$ with $\QQ_{\geq 0}S=\sigma^\vee$. For $\rho \in \sigma(1)$ we denote by $G_\rho$ the sublattice of $M \cap \rho^\perp$ generated by $S \cap \rho^\perp$ and by $\ell$ its index.
  
  Given an element $u \in M$ with $\langle u, v_\tau \rangle > 0$ for every $\tau \in \sigma(1)\setminus \{\rho\}$, where $v_\tau$ is the primitive generator of the ray $\tau$.  If $\ell$ is not divisible by $\operatorname{char}(k)$, then for the homogeneous component $(T^1_X)_u$ we have
  \[
    \dim_k (T^1_X)_u =
    \begin{cases}
      \ell - 2 & \ell  > 2,\, \langle u,\, v_\rho \rangle =-1 \\
      1 & \ell=2,\, \langle u,\, v_\rho \rangle =-2,\, u \in G_\rho-2e\\
      0 & \text{otherwise.}
    \end{cases}
  \]
  Here, $e\in M$ is any element with $\langle e, v_\rho \rangle =1$.
\end{proposition}
\begin{proof}
  By Lemma~\ref{lem:localisation} we obtain an isomorphism $(T^1_X)_u \cong (T^1_{U_\rho})_u$, where $U_\rho=\spec k[S + G_\rho]=\spec k[S + (-S \cap \rho^\perp)]$. Now, the claim follows from Lemma~\ref{lem:t1} below.
\end{proof}

\begin{lemma}
  \label{lem:t1}
  Let $G \subset \ZZ^{n}$ be a sublattice whose index $\ell$ is not divisible by $\operatorname{char}(k)$. Consider the seminormal toric variety $X=\spec k[S]$ with
  \[
    S= \{ (u,a) \in \ZZ^{n}\times \ZZ_{\geq 0} \mid u \in G \text{ or } a > 0\} \subset \ZZ^{n+1}=:M.
  \]
  Then
    \[
    \dim_k (T^1_X)_{(u,a)} =
    \begin{cases}
      \ell - 2 & \ell > 2,\, a =-1 \\
      1 & \ell=2,\, a =-2,\, u \in G\\
      0 & \text{otherwise.}
    \end{cases}
  \]
\end{lemma}
\begin{proof}
   We may assume that $G =\prod_{i=1}^n n_i \ZZ$. As a generating set of $S$ we may pick
  \[E=\underbrace{\{(\pm n_1e_1,0), \ldots, (\pm n_n e_n,0)\}}_{E'} \cup \underbrace{\{(u_1,\ldots,u_n,1) \mid \forall_i \colon 0\leq u_i \leq n_i-1\}}_{E''}.\]
  Then $E^{(u,a)}_S=\emptyset$ for $a \geq 1$ and $(T^1_X)_{(u,a)}=0$ according to Theorem~\ref{thm:t1-general-formula}.
  Similarly for $a=0$ and $u \in G$ we have $E^{(u,a)}_S = \emptyset$ and $(T^1_X)_{(u,a)}=0$.
  For $a=0$ and $u \notin G$ we have $E^{(u,a)}_S=\{n_1 e_1,\ldots,n_ne_n \}$ and
  $\Rel(E^{(u,a)}_S)$  is generated by the relations $q_i \colon 1\cdot (n_ie_i,0) + 1\cdot (-n_ie_i,0)=0$, but
  $\bar q_i =0$ and clearly $0 \notin G-u$. That means
  $\langle q \in \Rel(E^u_S) \mid \bar q + u \notin S \rangle=\Rel(E^{(u,a)}_S)$ and we obtain again $(T^1_X)_{(u,a)}=0$.

  We now distinguish several cases according to the values of $\ell$ and $a$.

  \textbf{Case $\ell=1$:}We have $E^{(u,-1)}_S=E'$ and $E^{(u,a)}_S=E$  for $a < -1$. In either case
   $\Rel(E^{u,a}_S)$  is generated by the relations $q_i \colon 1\cdot (n_ie_i,0) + 1\cdot (-n_ie_i,0)=0$ for $i=1,\ldots,n$
  with $\bar q_i + (u,a) =0 +(u,a) \notin S$, since $a <0$. Hence, as before we have $(T^1_X)_{(u,a)}=0$.

  \textbf{Case $\ell=2$:} We may assume that $n_1=2$ and $n_2=n_3=\ldots=n_n=1$.
  \textbf{Subcase $\ell=2$, $a=-1$:} Here 
  either
  \[E^{(u,a)}_S=\{(\pm 2e_1,0), (\pm e_2,0),\ldots, (\pm e_n,0), (e_1,1)\}\]
  holds or
  \[E^{(u,a)}_S=\{(\pm 2e_1,0), (\pm e_2,0), \ldots, (\pm e_n,0), (0,1)\},\]
  depending on whether $u \in G$ or not. In either case $\Rel(E^{(u,a)}_S)$
  is generated by the relations $q_i \colon 1\cdot(n_ie_i,0) + 1\cdot(-n_ie_i,0)=0$ for $i=1,\ldots,n$
  with $\bar q_i + (u,a) =0 +(u,a) \notin S$, since $a <0$. Hence, as before we have $(T^1_X)_{(u,a)}=0$.

  \textbf{Subcase $\ell=2$, $a=-2$:} Now, $E^{(u,a)}_S=E$ holds and  $\Rel(E^{(u,a)}_S)$ is generated by the relations $q_i \colon 1\cdot (n_ie_i,0) + 1\cdot (-n_ie_i,0)=0$ for $i=1,\ldots,n$ and additionally \[q \colon 1\cdot(2e_1,0) + 2\cdot (0,1) - 2 \cdot (e_1,1)=0\]
  Once again, we have $\bar q_i + (u,a) \notin S$, but we have $\bar q+(u,a)=(2e_1+u,-2) \in S$ if and only if $u \in 2\ZZ\times \ZZ^{n-1}$.
Hence, we have $\dim (T^1_x)_{(u,-2)} = 1$ if $u\in 2\ZZ\times \ZZ^{n-1}$ and $\dim (T^1_x)_{(u,-2)} = 0$ otherwise. 

\textbf{Subcase $\ell >2$, $a=-1$:}
Since $\{(u_1,\ldots,u_n) \mid \forall_i \colon 0\leq u_i \leq n_i-1\}$
contains exactly one representative for every class in $\ZZ^n/G$ there is exactly
one element $w \in \{(u_1,\ldots,u_n) \mid \forall_i \colon 0\leq u_i \leq n_i-1\}$
with $(w,1)+u \in S$. Hence, we have 
\[E^{(u,-1)}_S=E \setminus \{(w,1)\},\]
We have $\# E^{(u,-1)}_S = \ell + 2n - 1$ and $E^{(u,-1)}$ spans $M \otimes k$. Since $\dim M \otimes k = \operatorname{rk} M =n+1$ we get
$\dim \Rel_k(E^{(u,-1)}_S)=n+\ell-2$.
Now, the integral relation in $\Rel(E^{(u,-1)}_S)$ involve either none or at least $2$
elements from $E''$. The integral relation not involving any element from $E''$ are
again the ones generated freely by \[q_i \colon (n_ie_i,0) + (-n_ie_i,0)=0.\]
Let now $q$ be an integral relation from $\Rel(E^{(u,-1)}_S)$ involving at least two elements from $E''$
then $\bar q=(u',b)$ with $b \geq 2$. In particular, $\bar q+u \in S$.
Hence $\langle q \in \Rel(E^{(u,-1)}_S) \mid \bar q+u \notin S \rangle_k$ is of dimension $n$ with
basis consisting of the $q_i$ with $i=1,\ldots,n$ from above. We obtain
\[
\dim (T^1_X)_{(u,-1)}= \dim \Rel_k(E^{(u,-1)}_S)-\dim \langle q \mid \bar q+u \notin S \rangle_k =\ell-2.
\]

\textbf{Subcase $\ell>2$, $a=-2$:}
We have $E^{(u,a)}_S=E$ and $\Rel_k(E) \subset k^{E}$. We order the coordinates of $k^E$, such
that \[-n_1e_1 > \ldots > -n_ne_n > n_1e_1 > \ldots n_ne_n\] and the remaining ones are smaller and ordered weight-lexicographically.

We claim that for every
\[
  w \in E\setminus \left(\{(e_i,1) \mid i=1,\ldots,n\} \cup \{(0,1)\} \right)
\]
there is a relation $q(w) \in \langle q \in \Rel(E) \mid \bar q+u \notin S \rangle_k$ with leading non-zero entry being at the coordinate corresponding to $w$, i.e. $q(w)_v = 0$ for $v > w$ and $q(w)_w \neq 0$. It then follows that
\[\dim \langle q \in \Rel(E) \mid \bar q+u \notin S \rangle_k \geq \ell+n-1 =\dim \Rel_k(E).\]
Hence, we obtain $\langle q \in \Rel(E) \mid \bar q+u \notin S \rangle_k  =\dim \Rel_k(E)$ and therefore $T_X^1(u)=0$.

To show our claim we need to find appropriate relations $q(w)$.
For $w=-n_ie_i$ we take \[q(w) \colon 1\cdot(-n_ie_i) + 1 \cdot n_ie_i=0.\]
For $w=n_ie_i$ we may choose $q(w)$ as
\begin{equation*}
  \label{eq:q(w)-ei}
  1 \cdot (n_ie_i,0)+ 2 \cdot (0,1) + (-1)\cdot ((n_i-1)e_i,1) + (-1)\cdot (e_i,1)=0
\end{equation*}
if $\overline{q(w)}=(n_ie_i,2) \notin G-(u,-2)$. Otherwise we have $(n_ie_i+e_j,2) \notin G-(u,-2)$ and we may choose $q(w)$ as
\begin{equation*}
  \label{eq:q(w)-ei-irregular}
  1 \cdot (n_ie_i,0)+ 1\cdot (e_j,1) + 1 \cdot (0,1) + (-1)\cdot ((n_i-1)e_i,1) + (-1)\cdot (e_i+e_j,1)=0
\end{equation*}
with $j\neq i$. This is possible whenever there is at least one $j \neq i$ with $n_j>1$. On the other hand, if $n_j=1$ for all $j\neq i$ then $n_i=\ell > 2$ by our assumption. Hence, we may take $j=i$ and choose $q(w)$ as
\begin{equation*}
  \label{eq:q(w)-ei-n=1}
  1 \cdot (n_ie_i,0)+ 1\cdot (e_i,1) + 1\cdot(0,1) + (-1)\cdot ((n_i-1)e_i,1) + (-1)\cdot (2e_i,1)=0.
\end{equation*}

For the remaining choices of $w$ we have $w-(e_i,0) \in E \setminus \{(0,1)\}$ for some $i\in \{1,\ldots,n\}$.
Let us write $w=(w',1)$ and consider
\begin{equation*}
\label{eq:q(w)-regular}
q(w) \colon 1 \cdot (w',1) + 1\cdot(0,1) + (-1)\cdot (w'-e_i,1) +(-1)\cdot (e_i,1)=0.
\end{equation*}
If $\overline{q(w)}= (w',2) \notin \left(G-(u,-2)\right)$ it is clear that $q(w)$ is an element of $\langle q \mid \bar q+u \notin S \rangle_k$. Otherwise we have $(w'-e_i,2) \notin \left(G-(u,-2)\right)$ and also $(0,2) \notin \left(G-(u,-2)\right)$. Hence, the relation
\[ q \colon 1 \cdot (n_ie_i,0)+ 2 \cdot (0,1) + (-1)\cdot ((n_i-1)e_i,1) + (-1)\cdot (e_i,1)=0\]
as well as $q-q(w)$ are in $\langle q \in \Rel(E) \mid \bar q+u \notin S \rangle_k$ and, therefore, also $q(w)$. Indeed, 
$\bar q= (n_ie_i,2)$ and $\overline{q-q(w)}=((n_i-1)e_i+w',2)$.

\textbf{Subcase $\ell>1$, $a<-2$:}
We have $E^{(u,a)}_S=E$ with $\# E = \ell+2n$. The vector space of relations has then dimension
$\ell+n-1$. We have seen in the discussion of the previous case that  we have
$\ell+n-1$ linearly independent integral relations involving at most $2$ elements from $E''$. But all those elements now
lie in $\langle q \mid \bar q+u \notin S \rangle_k$. Hence, we must have
$\langle q \mid \bar q+u \notin S \rangle_k = \Rel_k(E^{(u,a)}_S)$ and, therefore, $(T^1_X)_{(u,a)}=0$.
\end{proof}

\begin{remark}
  Lemma~\ref{lem:t1} provides an explicit characterisation of the critical weights for a seminormal toric variety.
  Indeed, $u \in M$ is critical if and only if either $\langle u, v_\rho \rangle =-1$ for one $\rho \in \sigma(1)$ with $\ell_\rho > 2$ or $u \in G_\rho-2e$ for some $\rho \in \sigma(1)$ with $\ell_\rho=2$ and some $e \in M$ with $\langle e, v_\rho \rangle=1$.
\end{remark}

\begin{remark}
  To cover also the case when $\ell$ is divisible by $\operatorname{char}(k)=p$, we need to adapt the formulas for the dimensions in Proposition~\ref{prop:t1} and Lemma~\ref{lem:t1} as follows.
  \[
    \dim_k (T^1_X)_u =
    \begin{cases}
      \ell - 2 & \ell > 2,\, \langle u,\, v_\rho \rangle =-1 \\
      1 & \ell=2,\, \langle u,\, v_\rho \rangle =-2,\, u \in G_\rho-2e\\      
      \dim_k(G_\rho \otimes k) & \ell=2,\, \langle u,\, v_\rho \rangle =-1,\\
      \dim_k(G_\rho \otimes k)  & \langle u,\, v_\rho \rangle =0,\, u \notin G_\rho\\
      0 & \text{otherwise.}
    \end{cases}\]
  Indeed, the additional dimensions in this formula occur whenever $\# E^u_S \leq n+1$ (i.e. there are no non-trivial integral relations), but $E \cap \rho^\perp \subset E^u_S$. Those elements in the intersection which vanish modulo $p$ then provide linearly independent relations in $\Rel_k(E^u_S)/\langle q \mid \bar q +u \notin S \rangle_k$.
\end{remark}
\section{The case of seminormal surfaces}
\label{sec:surfaces}
In this section we apply our general results to the case of seminormal toric surfaces. Here, the $(S_2)$ property is automatically fulfilled and the situation is simple enough to explicitly determine dimensions of the homogeneous components of $T^1_X$. We restrict ourselves to the case of characteristic $0$.

Our setup is the same as  before. We consider a seminormal monoid $S\subset M \cong \ZZ^2$ and $\sigma=(\QQ_{\geq 0}S)^\vee$.  The case when $\sigma$ is one-dimensional is covered by Lemma~\ref{lem:t1}. Hence, we may assume that $\sigma$ is two-dimensional, spanned by two extremal rays $\rho_1$ and $\rho_2$ with primitive generators $v_1$ and $v_2$, respectively. The orders of the corresponding Galois groups are denoted by $\ell_1$ and $\ell_2$. We set $X=\spec k[S]$ and $\tilde X=\spec k [\bar S]$ is the normalisation. To simplify notation we set $S_i=S_{\rho_i}$, and $U_i=\spec k[S_i]$  for $i=1,2$. 

Consider a minimal generating set $\bar E=\{w_0,\ldots, w_{s+1}\}$ for $\bar S=\sigma^\vee \cap M$. If the cardinality of the generating set is $2$ (i.e. $\tilde X \cong \A^2$) we also add the sum of the two elements to the generating set in order to have at least three generators and allow for a unified treatment of this case. We order the elements $w_0,\ldots,w_{s+1}$ in the usual way, such that
$w_0 \in \rho_1^\perp$ and $w_{s+1} \in \rho_2^\perp$ and $w_{i-1}+w_{i+1}=a_i w_i$ holds for $i=1,\ldots,s$. Moreover, we set $g_1=w_0$, $g_2=w_{s+1}$. 

Let us number the lattice points on the line $\{ u \mid \langle v_i, u \rangle=j\}$
consecutively in such a way that the ones with positive numbers are exactly the ones
which pair positively with the other ray generator of $\sigma$. Let us write $u_{i}^{j}(m)$ for the $m$th lattice point within this numbering, i.e. $u_{i}^{j}(1)$ is the first lattice point on the line which lies in the interior of $\sigma^\vee$ if $j\neq 0$ and the first nonzero lattice point in $\sigma^\vee$ if $j=0$. We then have $u_1^0(m)=mw_0$, $u_2^0(m)=mw_{s+1}$, $u_1^1(m)=w_1+(m-1)w_0$,
$u_2^1(m)=w_s+(m-1)w_{s+1}$. See Figure~\ref{fig: u_i^j Notation} for an example.

\begin{figure}[h]
  \centering
  \begin{tikzpicture}[scale=0.65, >=stealth]
  \begin{scope}
    \clip (-2.6,-2.6) rectangle (5.9,7.9);
    \fill[blue!7] (0,0) -- (60,0) -- (60,180) -- cycle;
    \fill[red!7]  (0,0) -- (-60,0) -- (-60,-180) -- cycle;
    \draw[blue!55!black, thick] (0,0) -- (60,0);
    \draw[blue!55!black, thick] (0,0) -- (60,180);
    \draw[red!55!black, thick, opacity=.6]  (0,0) -- (-60,0);
    \draw[red!55!black, thick, opacity=.6]  (0,0) -- (-60,-180);
  \end{scope}
  \draw[->, gray!70] (-2.7,0) -- (6.0,0);
  \draw[->, gray!70] (0,-2.7) -- (0,8.0);
  \fill[gray!40] (-2,-2) circle (1.2pt);
  \fill[gray!40] (-2,-1) circle (1.2pt);
  \fill[gray!40] (-2,0) circle (1.2pt);
  \fill[gray!40] (-2,1) circle (1.2pt);
  \fill[gray!40] (-2,2) circle (1.2pt);
  \fill[gray!40] (-2,3) circle (1.2pt);
  \fill[gray!40] (-2,4) circle (1.2pt);
  \fill[gray!40] (-2,5) circle (1.2pt);
  \fill[gray!40] (-2,6) circle (1.2pt);
  \fill[gray!40] (-2,7) circle (1.2pt);
  \fill[gray!40] (-1,-2) circle (1.2pt);
  \fill[gray!40] (-1,-1) circle (1.2pt);
  \fill[gray!40] (-1,0) circle (1.2pt);
  \fill[gray!40] (-1,1) circle (1.2pt);
  \fill[gray!40] (-1,2) circle (1.2pt);
  \fill[gray!40] (-1,3) circle (1.2pt);
  \fill[gray!40] (-1,4) circle (1.2pt);
  \fill[gray!40] (-1,5) circle (1.2pt);
  \fill[gray!40] (-1,6) circle (1.2pt);
  \fill[gray!40] (-1,7) circle (1.2pt);
  \fill[gray!40] (0,-2) circle (1.2pt);
  \fill[gray!40] (0,-1) circle (1.2pt);
  \fill[gray!40] (0,1) circle (1.2pt);
  \fill[gray!40] (0,2) circle (1.2pt);
  \fill[gray!40] (0,3) circle (1.2pt);
  \fill[gray!40] (0,4) circle (1.2pt);
  \fill[gray!40] (0,5) circle (1.2pt);
  \fill[gray!40] (0,6) circle (1.2pt);
  \fill[gray!40] (0,7) circle (1.2pt);
  \fill[gray!40] (1,-2) circle (1.2pt);
  \fill[gray!40] (1,-1) circle (1.2pt);
  \fill[blue!70] (1,0) circle (2.6pt);
  \fill[blue!70] (1,1) circle (2.6pt);
  \fill[blue!70] (1,2) circle (2.6pt);
  \fill[blue!70] (1,3) circle (2.6pt);
  \fill[gray!40] (1,4) circle (1.2pt);
  \fill[gray!40] (1,5) circle (1.2pt);
  \fill[gray!40] (1,6) circle (1.2pt);
  \fill[gray!40] (1,7) circle (1.2pt);
  \fill[gray!40] (2,-2) circle (1.2pt);
  \fill[gray!40] (2,-1) circle (1.2pt);
  \fill[blue!70] (2,0) circle (2.6pt);
  \fill[blue!70] (2,1) circle (2.6pt);
  \fill[blue!70] (2,2) circle (2.6pt);
  \fill[blue!70] (2,3) circle (2.6pt);
  \fill[blue!70] (2,4) circle (2.6pt);
  \fill[blue!70] (2,5) circle (2.6pt);
  \fill[blue!70] (2,6) circle (2.6pt);
  \fill[gray!40] (2,7) circle (1.2pt);
  \fill[gray!40] (3,-2) circle (1.2pt);
  \fill[gray!40] (3,-1) circle (1.2pt);
  \fill[blue!70] (3,0) circle (2.6pt);
  \fill[blue!70] (3,1) circle (2.6pt);
  \fill[blue!70] (3,2) circle (2.6pt);
  \fill[blue!70] (3,3) circle (2.6pt);
  \fill[blue!70] (3,4) circle (2.6pt);
  \fill[blue!70] (3,5) circle (2.6pt);
  \fill[blue!70] (3,6) circle (2.6pt);
  \fill[blue!70] (3,7) circle (2.6pt);
  \fill[gray!40] (4,-2) circle (1.2pt);
  \fill[gray!40] (4,-1) circle (1.2pt);
  \fill[blue!70] (4,0) circle (2.6pt);
  \fill[blue!70] (4,1) circle (2.6pt);
  \fill[blue!70] (4,2) circle (2.6pt);
  \fill[blue!70] (4,3) circle (2.6pt);
  \fill[blue!70] (4,4) circle (2.6pt);
  \fill[blue!70] (4,5) circle (2.6pt);
  \fill[blue!70] (4,6) circle (2.6pt);
  \fill[blue!70] (4,7) circle (2.6pt);
  \fill[gray!40] (5,-2) circle (1.2pt);
  \fill[gray!40] (5,-1) circle (1.2pt);
  \fill[blue!70] (5,0) circle (2.6pt);
  \fill[blue!70] (5,1) circle (2.6pt);
  \fill[blue!70] (5,2) circle (2.6pt);
  \fill[blue!70] (5,3) circle (2.6pt);
  \fill[blue!70] (5,4) circle (2.6pt);
  \fill[blue!70] (5,5) circle (2.6pt);
  \fill[blue!70] (5,6) circle (2.6pt);
  \fill[blue!70] (5,7) circle (2.6pt);
  \draw[blue!65, dashed, thick] (-2.4,1) -- (5.8,1);
  \node[blue!65!black, font=\scriptsize, anchor=west] at (5.8,1) {$u_1^1(m)$};
  \node[blue!65!black, font=\scriptsize, anchor=west] at (5.8,0) {$u_1^0(m)$};
  \fill[gray!55!black] (-1,1) circle (2.4pt);
  \node[gray!55!black, font=\scriptsize, anchor=south, yshift=2.5pt] at (-1,1) {$-1$};
  \fill[gray!55!black] (0,1) circle (2.4pt);
  \node[gray!55!black, font=\scriptsize, anchor=south, yshift=2.5pt] at (0,1) {$0$};
  \fill[blue!70!black] (1,1) circle (2.4pt);
  \node[blue!70!black, font=\scriptsize, anchor=south, yshift=2.5pt] at (1,1) {$1$};
  \fill[blue!70!black] (2,1) circle (2.4pt);
  \node[blue!70!black, font=\scriptsize, anchor=south, yshift=2.5pt] at (2,1) {$2$};
  \fill[blue!70!black] (3,1) circle (2.4pt);
  \node[blue!70!black, font=\scriptsize, anchor=south, yshift=2.5pt] at (3,1) {$3$};
  \fill[blue!70!black] (4,1) circle (2.4pt);
  \node[blue!70!black, font=\scriptsize, anchor=south, yshift=2.5pt] at (4,1) {$4$};
  \fill[blue!70!black] (1,0) circle (2.4pt);
  \node[blue!70!black, font=\scriptsize, anchor=north, yshift=-2.5pt] at (1,0) {$1$};
  \fill[blue!70!black] (2,0) circle (2.4pt);
  \node[blue!70!black, font=\scriptsize, anchor=north, yshift=-2.5pt] at (2,0) {$2$};
  \fill[blue!70!black] (3,0) circle (2.4pt);
  \node[blue!70!black, font=\scriptsize, anchor=north, yshift=-2.5pt] at (3,0) {$3$};
  \fill[blue!70!black] (4,0) circle (2.4pt);
  \node[blue!70!black, font=\scriptsize, anchor=north, yshift=-2.5pt] at (4,0) {$4$};
  \draw[red!65, dashed, thick] (-0.5,-2.5) -- (2.7,7.1);
  \fill[gray!55!black] (0,-1) circle (2.4pt);
  \node[gray!55!black, font=\scriptsize, anchor=west, xshift=4pt, yshift=-4pt] at (0,-1) {$0$};
  \fill[red!70!black] (1,2) circle (2.4pt);
  \node[red!70!black, font=\scriptsize, anchor=west, xshift=4pt, yshift=-4pt] at (1,2) {$1$};
  \fill[red!70!black] (2,5) circle (2.4pt);
  \node[red!70!black, font=\scriptsize, anchor=west, xshift=4pt, yshift=-4pt] at (2,5) {$2$};
  \fill[red!70!black] (1,3) circle (2.4pt);
  \node[red!70!black, font=\scriptsize, anchor=east, xshift=-3pt] at (1,3) {$1$};
  \fill[red!70!black] (2,6) circle (2.4pt);
  \node[red!70!black, font=\scriptsize, anchor=east, xshift=-3pt] at (2,6) {$2$};
  \node[red!65!black, font=\scriptsize, anchor=south west] at (2.6,7.1) {$u_2^1(m)$};
  \node[red!65!black, font=\scriptsize, anchor=south east, xshift=-2pt] at (2,6.15) {$u_2^0(m)$};

\end{tikzpicture}
\caption{The lattice points $u_i^j(m)$ for $i=1,2$ and $j=0,1$.}
\label{fig: u_i^j Notation}
\end{figure}
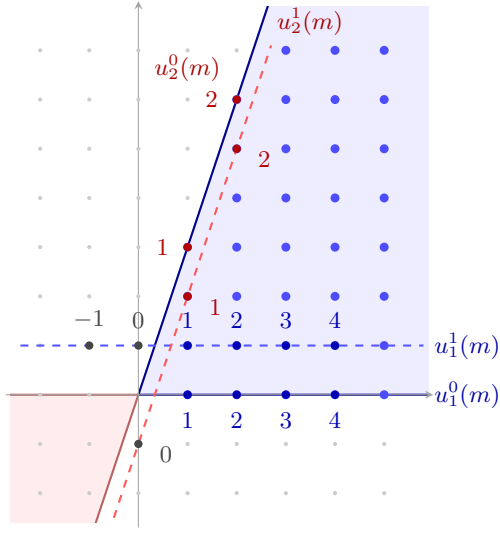

The aim of this section is to prove the following theorem.
\begin{theorem}
\label{thm:T1dim-surface}
In our setup the  dimensions of the homogeneous components $(T^1_{X})_u$  are given as follows.
      \begin{equation}
        \dim(T^1_{X})_u= \max \{0,d_1(u)+d_2(u)+\tilde{d}(u)\}\label{eq:dimT1-surfaces}
      \end{equation}
      with
      \[d_i(u)=
        \begin{cases}
          1 & u=-u_{i}^{0}(m),\, 1\leq m < \ell_i\\
          \ell_i-2 & u=-u_{i}^{1}(m),\, m \leq 0\\
          \ell_i-1 & u=-u_{i}^{1}(1)\\
          \ell_i-m + 1& u=-u_{i}^{1}(m),\, 2\leq m \leq \ell_i\\
          1 & u=-2u_{i}^{1}(m),\, m \leq 1,\, \ell_i=2\\
          0 & \text{otherwise}
        \end{cases}
      \]
      for $i=1,2$ and
      \[
        \tilde d(u)=
        \begin{cases}
          1 & u=-mw_i,\, 2\leq m < a_i,\, 1\leq i \leq s\\
          1 & u=-a_1w_1, s=1, a_1 > 1\\
          2 - \delta_{i1}-\delta_{is} & u=-w_i,\, 1\leq i \leq s\\
          0 & \text{otherwise.}
        \end{cases}
      \]
      For $\tilde X \not \cong \A^2$ we have the following exception
       \begin{enumerate}[a)]
       \item  If $s=1$, $u=-a_1w_1$ and $\ell_1=1,\ell_2 \neq 1$ or
         $\ell_1 \neq 1,\ell_2 = 1$, then
         $\dim (T^1_X)_u$ decreases by $1$ compared to (\ref{eq:dimT1-surfaces}).
         \label{item:singular-exception-multiple}
       \end{enumerate}
       For $\tilde X \cong \A^2$ the following additional rules apply
           \begin{enumerate}[a)]
             \setcounter{enumi}{1}
           \item If $\ell_1,\ell_2 \geq 2$ and $u=-u_{i}^{1}(\ell_i)$, then $\dim (T^1_X)_{u}=0$ holds for $i=1,2$.\label{item:smooth-exception-l2}
           \item If we have $\ell_1=1$ and $u=-u_{2}^{1}(m)$ with $1 \leq m \leq \ell_2-1$ or  $\ell_2=1$ and $u=-u_{1}^{1}(m)$ with $1 \leq m \leq \ell_1-1$, then $\dim (T_X^1)_{u}$ is decreased by $1$ compared to (\ref{eq:dimT1-surfaces}).    \label{item:smooth-exception-li=1}
            \item In degree $u=-2u_{1}^{1}(1)=-2u_{2}^{1}(1)$ we have \[\dim(T_X^1)_u=
                 \begin{cases}
                   1 & \ell_1,\ell_2\leq 2\text{ and } (\ell_1,\ell_2)\neq (1,1)\\
                   0 & \text{otherwise.}
                 \end{cases}\]\label{item:smooth-exception-l1=l2=2}
            \end{enumerate}            
\end{theorem}

\begin{example}
  Consider the case, when $\sigma$ is spanned by elements of a lattice basis, i.e.  $\tilde X \cong \A^2$. We have $\tilde d \equiv 0$. For most values of
  $u$ at most one of the contributions $d_1$ and $d_2$ is non-trivial. The only exceptions are $u=-u_1^1(1)=-u_2^1(1)$
  with $\dim (T^1_X)_u=\ell_1+\ell_2-2$,   for $u=-u_1^0(1)=-u_2^1(0)$ with $\dim(T^1_X)_u=\ell_2-1$ if $\ell_1 \geq 2$ holds and for $u=-u_2^0(1)=-u_1^1(0)$ with $\dim(T^1_X)_u=\ell_1-1$ if $\ell_2 \geq 2$ holds, see Figure~\ref{fig:smooth-case}. For $\ell_1=1$ and $\ell_2 \geq 3$ Exception~\ref{item:smooth-exception-li=1} applies, see Figure~\ref{fig:smooth-exceptions} for a concrete example.
\begin{figure}[h]
  \centering
  \begin{subfigure}[t]{0.45\textwidth}
    \begin{tikzpicture}[scale=0.47, >=stealth]
      \begin{scope}
        \clip (-5.6,-5.6) rectangle (5.9,5.9);
        \fill[blue!7] (0,0) -- (60,0) -- (0,60) -- cycle;
        \fill[red!7]  (0,0) -- (-60,0) -- (0,-60) -- cycle;
        \draw[blue!55!black, thick] (0,0) -- (60,0);
        \draw[blue!55!black, thick] (0,0) -- (0,60);
        \draw[red!55!black, thick, opacity=.6]  (0,0) -- (-60,0);
        \draw[red!55!black, thick, opacity=.6]  (0,0) -- (0,-60);
      \end{scope}
      \draw[->, gray!70] (-5.7,0) -- (6.0,0);
      \draw[->, gray!70] (0,-5.7) -- (0,6.0);
      \fill[gray!40] (-5,-5) circle (1.2pt);
      \fill[gray!40] (-5,-4) circle (1.2pt);
      \fill[gray!40] (-5,-3) circle (1.2pt);
      \fill[gray!40] (-5,-2) circle (1.2pt);
      \fill[gray!40] (-5,-1) circle (1.2pt);
      \fill[gray!40] (-5,0) circle (1.2pt);
      \fill[gray!40] (-5,1) circle (1.2pt);
      \fill[gray!40] (-5,2) circle (1.2pt);
      \fill[gray!40] (-5,3) circle (1.2pt);
      \fill[gray!40] (-5,4) circle (1.2pt);
      \fill[gray!40] (-5,5) circle (1.2pt);
      \fill[gray!40] (-4,-5) circle (1.2pt);
      \fill[gray!40] (-4,-4) circle (1.2pt);
      \fill[gray!40] (-4,-3) circle (1.2pt);
      \fill[gray!40] (-4,-2) circle (1.2pt);
      \fill[gray!40] (-4,-1) circle (1.2pt);
      \fill[gray!40] (-4,0) circle (1.2pt);
      \fill[gray!40] (-4,1) circle (1.2pt);
      \fill[gray!40] (-4,2) circle (1.2pt);
      \fill[gray!40] (-4,3) circle (1.2pt);
      \fill[gray!40] (-4,4) circle (1.2pt);
      \fill[gray!40] (-4,5) circle (1.2pt);
      \fill[gray!40] (-3,-5) circle (1.2pt);
      \fill[gray!40] (-3,-4) circle (1.2pt);
      \fill[gray!40] (-3,-3) circle (1.2pt);
      \fill[gray!40] (-3,-2) circle (1.2pt);
      \fill[gray!40] (-3,-1) circle (1.2pt);
      \fill[gray!40] (-3,0) circle (1.2pt);
      \fill[gray!40] (-3,1) circle (1.2pt);
      \fill[gray!40] (-3,2) circle (1.2pt);
      \fill[gray!40] (-3,3) circle (1.2pt);
      \fill[gray!40] (-3,4) circle (1.2pt);
      \fill[gray!40] (-3,5) circle (1.2pt);
      \fill[gray!40] (-2,-5) circle (1.2pt);
      \fill[gray!40] (-2,-4) circle (1.2pt);
      \fill[gray!40] (-2,-3) circle (1.2pt);
      \fill[gray!40] (-2,-2) circle (1.2pt);
      \fill[gray!40] (-2,-1) circle (1.2pt);
      \fill[gray!40] (-2,0) circle (1.2pt);
      \fill[gray!40] (-2,1) circle (1.2pt);
      \fill[gray!40] (-2,2) circle (1.2pt);
      \fill[gray!40] (-2,3) circle (1.2pt);
      \fill[gray!40] (-2,4) circle (1.2pt);
      \fill[gray!40] (-2,5) circle (1.2pt);
      \fill[gray!40] (-1,-5) circle (1.2pt);
      \fill[gray!40] (-1,-4) circle (1.2pt);
      \fill[gray!40] (-1,-3) circle (1.2pt);
      \fill[gray!40] (-1,-2) circle (1.2pt);
      \fill[gray!40] (-1,-1) circle (1.2pt);
      \fill[gray!40] (-1,0) circle (1.2pt);
      \fill[gray!40] (-1,1) circle (1.2pt);
      \fill[gray!40] (-1,2) circle (1.2pt);
      \fill[gray!40] (-1,3) circle (1.2pt);
      \fill[gray!40] (-1,4) circle (1.2pt);
      \fill[gray!40] (-1,5) circle (1.2pt);
      \fill[gray!40] (0,-5) circle (1.2pt);
      \fill[gray!40] (0,-4) circle (1.2pt);
      \fill[gray!40] (0,-3) circle (1.2pt);
      \fill[gray!40] (0,-2) circle (1.2pt);
      \fill[gray!40] (0,-1) circle (1.2pt);
      \fill[gray!40] (0,1) circle (1.2pt);
      \fill[gray!40] (0,2) circle (1.2pt);
      \fill[gray!40] (0,3) circle (1.2pt);
      \fill[gray!40] (0,4) circle (1.2pt);
      \fill[blue!70] (0,5) circle (2.6pt);
      \fill[gray!40] (1,-5) circle (1.2pt);
      \fill[gray!40] (1,-4) circle (1.2pt);
      \fill[gray!40] (1,-3) circle (1.2pt);
      \fill[gray!40] (1,-2) circle (1.2pt);
      \fill[gray!40] (1,-1) circle (1.2pt);
      \fill[gray!40] (1,0) circle (1.2pt);
      \fill[blue!70] (1,1) circle (2.6pt);
      \fill[blue!70] (1,2) circle (2.6pt);
      \fill[blue!70] (1,3) circle (2.6pt);
      \fill[blue!70] (1,4) circle (2.6pt);
      \fill[blue!70] (1,5) circle (2.6pt);
      \fill[gray!40] (2,-5) circle (1.2pt);
      \fill[gray!40] (2,-4) circle (1.2pt);
      \fill[gray!40] (2,-3) circle (1.2pt);
      \fill[gray!40] (2,-2) circle (1.2pt);
      \fill[gray!40] (2,-1) circle (1.2pt);
      \fill[gray!40] (2,0) circle (1.2pt);
      \fill[blue!70] (2,1) circle (2.6pt);
      \fill[blue!70] (2,2) circle (2.6pt);
      \fill[blue!70] (2,3) circle (2.6pt);
      \fill[blue!70] (2,4) circle (2.6pt);
      \fill[blue!70] (2,5) circle (2.6pt);
      \fill[gray!40] (3,-5) circle (1.2pt);
      \fill[gray!40] (3,-4) circle (1.2pt);
      \fill[gray!40] (3,-3) circle (1.2pt);
      \fill[gray!40] (3,-2) circle (1.2pt);
      \fill[gray!40] (3,-1) circle (1.2pt);
      \fill[gray!40] (3,0) circle (1.2pt);
      \fill[blue!70] (3,1) circle (2.6pt);
      \fill[blue!70] (3,2) circle (2.6pt);
      \fill[blue!70] (3,3) circle (2.6pt);
      \fill[blue!70] (3,4) circle (2.6pt);
      \fill[blue!70] (3,5) circle (2.6pt);
      \fill[gray!40] (4,-5) circle (1.2pt);
      \fill[gray!40] (4,-4) circle (1.2pt);
      \fill[gray!40] (4,-3) circle (1.2pt);
      \fill[gray!40] (4,-2) circle (1.2pt);
      \fill[gray!40] (4,-1) circle (1.2pt);
      \fill[blue!70] (4,0) circle (2.6pt);
      \fill[blue!70] (4,1) circle (2.6pt);
      \fill[blue!70] (4,2) circle (2.6pt);
      \fill[blue!70] (4,3) circle (2.6pt);
      \fill[blue!70] (4,4) circle (2.6pt);
      \fill[blue!70] (4,5) circle (2.6pt);
      \fill[gray!40] (5,-5) circle (1.2pt);
      \fill[gray!40] (5,-4) circle (1.2pt);
      \fill[gray!40] (5,-3) circle (1.2pt);
      \fill[gray!40] (5,-2) circle (1.2pt);
      \fill[gray!40] (5,-1) circle (1.2pt);
      \fill[gray!40] (5,0) circle (1.2pt);
      \fill[blue!70] (5,1) circle (2.6pt);
      \fill[blue!70] (5,2) circle (2.6pt);
      \fill[blue!70] (5,3) circle (2.6pt);
      \fill[blue!70] (5,4) circle (2.6pt);
      \fill[blue!70] (5,5) circle (2.6pt);
      \draw[orange!85!black, line width=1.3pt, opacity=.30] (-5.5,-1) -- (5.6,-1);
      \draw[teal!85!black, line width=1.3pt, opacity=.30] (-1,-5.5) -- (-1,5.6);
      \draw[red!60!black, line width=1.3pt, opacity=.35] (-5.5,0) -- (0,0);
      \draw[red!60!black, line width=1.3pt, opacity=.35] (0,-5.5) -- (0,0);
      \node[fill=white, inner sep=1pt, font=\footnotesize] at (-3,-1) {$2$};
      \node[fill=white, inner sep=1pt, font=\footnotesize] at (-3,0) {$1$};
      \node[fill=white, inner sep=1pt, font=\footnotesize] at (-2,-1) {$3$};
      \node[fill=white, inner sep=1pt, font=\footnotesize] at (-2,0) {$1$};
      \node[fill=white, inner sep=1pt, font=\footnotesize] at (-1,-4) {$2$};
      \node[fill=white, inner sep=1pt, font=\footnotesize] at (-1,-3) {$3$};
      \node[fill=white, inner sep=1pt, font=\footnotesize] at (-1,-2) {$4$};
      \node[fill=yellow!45, draw=black, line width=.4pt, rounded corners=1pt, inner sep=1.6pt, font=\footnotesize\bfseries] at (-1,-1) {$7$};
      \node[fill=yellow!45, draw=black, line width=.4pt, rounded corners=1pt, inner sep=1.6pt, font=\footnotesize\bfseries] at (-1,0) {$4$};
      \node[fill=white, inner sep=1pt, font=\footnotesize] at (-1,1) {$3$};
      \node[fill=white, inner sep=1pt, font=\footnotesize] at (-1,2) {$3$};
      \node[fill=white, inner sep=1pt, font=\footnotesize] at (-1,3) {$3$};
      \node[fill=white, inner sep=1pt, font=\footnotesize] at (-1,4) {$3$};
      \node[fill=white, inner sep=1pt, font=\footnotesize] at (-1,5) {$3$};
      \node[fill=white, inner sep=1pt, font=\footnotesize] at (0,-4) {$1$};
      \node[fill=white, inner sep=1pt, font=\footnotesize] at (0,-3) {$1$};
      \node[fill=white, inner sep=1pt, font=\footnotesize] at (0,-2) {$1$};
      \node[fill=yellow!45, draw=black, line width=.4pt, rounded corners=1pt, inner sep=1.6pt, font=\footnotesize\bfseries] at (0,-1) {$3$};
      \node[fill=white, inner sep=1pt, font=\footnotesize] at (1,-1) {$2$};
      \node[fill=white, inner sep=1pt, font=\footnotesize] at (2,-1) {$2$};
      \node[fill=white, inner sep=1pt, font=\footnotesize] at (3,-1) {$2$};
      \node[fill=white, inner sep=1pt, font=\footnotesize] at (4,-1) {$2$};
      \node[fill=white, inner sep=1pt, font=\footnotesize] at (5,-1) {$2$};
      \node[font=\scriptsize] at (5.75,-1) {$\cdots$};
      \node[font=\scriptsize] at (-1,5.75) {$\vdots$};
    \end{tikzpicture}
    \caption{$\ell_1=4,\ell_2=5$}
  \end{subfigure}  \begin{subfigure}[t]{0.45\textwidth}
    \begin{tikzpicture}[scale=0.47, >=stealth]
      \begin{scope}
        \clip (-5.6,-5.6) rectangle (5.9,5.9);
        \fill[blue!7] (0,0) -- (60,0) -- (0,60) -- cycle;
        \fill[red!7]  (0,0) -- (-60,0) -- (0,-60) -- cycle;
        \draw[blue!55!black, thick] (0,0) -- (60,0);
        \draw[blue!55!black, thick] (0,0) -- (0,60);
        \draw[red!55!black, thick, opacity=.6]  (0,0) -- (-60,0);
        \draw[red!55!black, thick, opacity=.6]  (0,0) -- (0,-60);
      \end{scope}
      \draw[->, gray!70] (-5.7,0) -- (6.0,0);
      \draw[->, gray!70] (0,-5.7) -- (0,6.0);
      \fill[gray!40] (-5,-5) circle (1.2pt);
      \fill[gray!40] (-5,-4) circle (1.2pt);
      \fill[gray!40] (-5,-3) circle (1.2pt);
      \fill[gray!40] (-5,-2) circle (1.2pt);
      \fill[gray!40] (-5,-1) circle (1.2pt);
      \fill[gray!40] (-5,0) circle (1.2pt);
      \fill[gray!40] (-5,1) circle (1.2pt);
      \fill[gray!40] (-5,2) circle (1.2pt);
      \fill[gray!40] (-5,3) circle (1.2pt);
      \fill[gray!40] (-5,4) circle (1.2pt);
      \fill[gray!40] (-5,5) circle (1.2pt);
      \fill[gray!40] (-4,-5) circle (1.2pt);
      \fill[gray!40] (-4,-4) circle (1.2pt);
      \fill[gray!40] (-4,-3) circle (1.2pt);
      \fill[gray!40] (-4,-2) circle (1.2pt);
      \fill[gray!40] (-4,-1) circle (1.2pt);
      \fill[gray!40] (-4,0) circle (1.2pt);
      \fill[gray!40] (-4,1) circle (1.2pt);
      \fill[gray!40] (-4,2) circle (1.2pt);
      \fill[gray!40] (-4,3) circle (1.2pt);
      \fill[gray!40] (-4,4) circle (1.2pt);
      \fill[gray!40] (-4,5) circle (1.2pt);
      \fill[gray!40] (-3,-5) circle (1.2pt);
      \fill[gray!40] (-3,-4) circle (1.2pt);
      \fill[gray!40] (-3,-3) circle (1.2pt);
      \fill[gray!40] (-3,-2) circle (1.2pt);
      \fill[gray!40] (-3,-1) circle (1.2pt);
      \fill[gray!40] (-3,0) circle (1.2pt);
      \fill[gray!40] (-3,1) circle (1.2pt);
      \fill[gray!40] (-3,2) circle (1.2pt);
      \fill[gray!40] (-3,3) circle (1.2pt);
      \fill[gray!40] (-3,4) circle (1.2pt);
      \fill[gray!40] (-3,5) circle (1.2pt);
      \fill[gray!40] (-2,-5) circle (1.2pt);
      \fill[gray!40] (-2,-4) circle (1.2pt);
      \fill[gray!40] (-2,-3) circle (1.2pt);
      \fill[gray!40] (-2,-2) circle (1.2pt);
      \fill[gray!40] (-2,-1) circle (1.2pt);
      \fill[gray!40] (-2,0) circle (1.2pt);
      \fill[gray!40] (-2,1) circle (1.2pt);
      \fill[gray!40] (-2,2) circle (1.2pt);
      \fill[gray!40] (-2,3) circle (1.2pt);
      \fill[gray!40] (-2,4) circle (1.2pt);
      \fill[gray!40] (-2,5) circle (1.2pt);
      \fill[gray!40] (-1,-5) circle (1.2pt);
      \fill[gray!40] (-1,-4) circle (1.2pt);
      \fill[gray!40] (-1,-3) circle (1.2pt);
      \fill[gray!40] (-1,-2) circle (1.2pt);
      \fill[gray!40] (-1,-1) circle (1.2pt);
      \fill[gray!40] (-1,0) circle (1.2pt);
      \fill[gray!40] (-1,1) circle (1.2pt);
      \fill[gray!40] (-1,2) circle (1.2pt);
      \fill[gray!40] (-1,3) circle (1.2pt);
      \fill[gray!40] (-1,4) circle (1.2pt);
      \fill[gray!40] (-1,5) circle (1.2pt);
      \fill[gray!40] (0,-5) circle (1.2pt);
      \fill[gray!40] (0,-4) circle (1.2pt);
      \fill[gray!40] (0,-3) circle (1.2pt);
      \fill[gray!40] (0,-2) circle (1.2pt);
      \fill[gray!40] (0,-1) circle (1.2pt);
      \fill[gray!40] (0,1) circle (1.2pt);
      \fill[gray!40] (0,2) circle (1.2pt);
      \fill[gray!40] (0,3) circle (1.2pt);
      \fill[gray!40] (0,4) circle (1.2pt);
      \fill[blue!70] (0,5) circle (2.6pt);
      \fill[gray!40] (1,-5) circle (1.2pt);
      \fill[gray!40] (1,-4) circle (1.2pt);
      \fill[gray!40] (1,-3) circle (1.2pt);
      \fill[gray!40] (1,-2) circle (1.2pt);
      \fill[gray!40] (1,-1) circle (1.2pt);
      \fill[blue!70] (1,0) circle (2.6pt);
      \fill[blue!70] (1,1) circle (2.6pt);
      \fill[blue!70] (1,2) circle (2.6pt);
      \fill[blue!70] (1,3) circle (2.6pt);
      \fill[blue!70] (1,4) circle (2.6pt);
      \fill[blue!70] (1,5) circle (2.6pt);
      \fill[gray!40] (2,-5) circle (1.2pt);
      \fill[gray!40] (2,-4) circle (1.2pt);
      \fill[gray!40] (2,-3) circle (1.2pt);
      \fill[gray!40] (2,-2) circle (1.2pt);
      \fill[gray!40] (2,-1) circle (1.2pt);
      \fill[blue!70] (2,0) circle (2.6pt);
      \fill[blue!70] (2,1) circle (2.6pt);
      \fill[blue!70] (2,2) circle (2.6pt);
      \fill[blue!70] (2,3) circle (2.6pt);
      \fill[blue!70] (2,4) circle (2.6pt);
      \fill[blue!70] (2,5) circle (2.6pt);
      \fill[gray!40] (3,-5) circle (1.2pt);
      \fill[gray!40] (3,-4) circle (1.2pt);
      \fill[gray!40] (3,-3) circle (1.2pt);
      \fill[gray!40] (3,-2) circle (1.2pt);
      \fill[gray!40] (3,-1) circle (1.2pt);
      \fill[blue!70] (3,0) circle (2.6pt);
      \fill[blue!70] (3,1) circle (2.6pt);
      \fill[blue!70] (3,2) circle (2.6pt);
      \fill[blue!70] (3,3) circle (2.6pt);
      \fill[blue!70] (3,4) circle (2.6pt);
      \fill[blue!70] (3,5) circle (2.6pt);
      \fill[gray!40] (4,-5) circle (1.2pt);
      \fill[gray!40] (4,-4) circle (1.2pt);
      \fill[gray!40] (4,-3) circle (1.2pt);
      \fill[gray!40] (4,-2) circle (1.2pt);
      \fill[gray!40] (4,-1) circle (1.2pt);
      \fill[blue!70] (4,0) circle (2.6pt);
      \fill[blue!70] (4,1) circle (2.6pt);
      \fill[blue!70] (4,2) circle (2.6pt);
      \fill[blue!70] (4,3) circle (2.6pt);
      \fill[blue!70] (4,4) circle (2.6pt);
      \fill[blue!70] (4,5) circle (2.6pt);
      \fill[gray!40] (5,-5) circle (1.2pt);
      \fill[gray!40] (5,-4) circle (1.2pt);
      \fill[gray!40] (5,-3) circle (1.2pt);
      \fill[gray!40] (5,-2) circle (1.2pt);
      \fill[gray!40] (5,-1) circle (1.2pt);
      \fill[blue!70] (5,0) circle (2.6pt);
      \fill[blue!70] (5,1) circle (2.6pt);
      \fill[blue!70] (5,2) circle (2.6pt);
      \fill[blue!70] (5,3) circle (2.6pt);
      \fill[blue!70] (5,4) circle (2.6pt);
      \fill[blue!70] (5,5) circle (2.6pt);
      \draw[orange!85!black, line width=1.3pt, opacity=.30] (-5.5,-1) -- (5.6,-1);
      \draw[teal!85!black, line width=1.3pt, opacity=.30] (-1,-5.5) -- (-1,5.6);
      \draw[red!60!black, line width=1.3pt, opacity=.35] (-5.5,0) -- (0,0);
      \draw[red!60!black, line width=1.3pt, opacity=.35] (0,-5.5) -- (0,0);
      \node[fill=white, inner sep=1pt, font=\footnotesize] at (-1,-5) {$1$};
      \node[fill=white, inner sep=1pt, font=\footnotesize] at (-1,-4) {$1$};
      \node[fill=white, inner sep=1pt, font=\footnotesize] at (-1,-3) {$2$};
      \node[fill=white, inner sep=1pt, font=\footnotesize] at (-1,-2) {$3$};
      \node[fill=yellow!45, draw=black, line width=.4pt, rounded corners=1pt, inner sep=1.6pt, font=\footnotesize\bfseries] at (-1,-1) {$3$};
      \node[fill=white, inner sep=1pt, font=\footnotesize] at (-1,0) {$3$};
      \node[fill=white, inner sep=1pt, font=\footnotesize] at (-1,1) {$3$};
      \node[fill=white, inner sep=1pt, font=\footnotesize] at (-1,2) {$3$};
      \node[fill=white, inner sep=1pt, font=\footnotesize] at (-1,3) {$3$};
      \node[fill=white, inner sep=1pt, font=\footnotesize] at (-1,4) {$3$};
      \node[fill=white, inner sep=1pt, font=\footnotesize] at (-1,5) {$3$};
      \node[fill=white, inner sep=1pt, font=\footnotesize] at (0,-4) {$1$};
      \node[fill=white, inner sep=1pt, font=\footnotesize] at (0,-3) {$1$};
      \node[fill=white, inner sep=1pt, font=\footnotesize] at (0,-2) {$1$};
      \node[fill=yellow!45, draw=black, line width=.4pt, rounded corners=1pt, inner sep=1.6pt, font=\footnotesize\bfseries] at (0,-1) {$0$};
      \node[font=\scriptsize] at (-1,5.75) {$\vdots$};
    \end{tikzpicture}
    \caption{$\ell_1=1,\ell_2=5$}
    \label{fig:smooth-exceptions}
  \end{subfigure}
 \caption{Dimensions of $(T^1_X)_u$ for the case $\tilde X=\A^2$.}
\label{fig:smooth-case}
\end{figure}
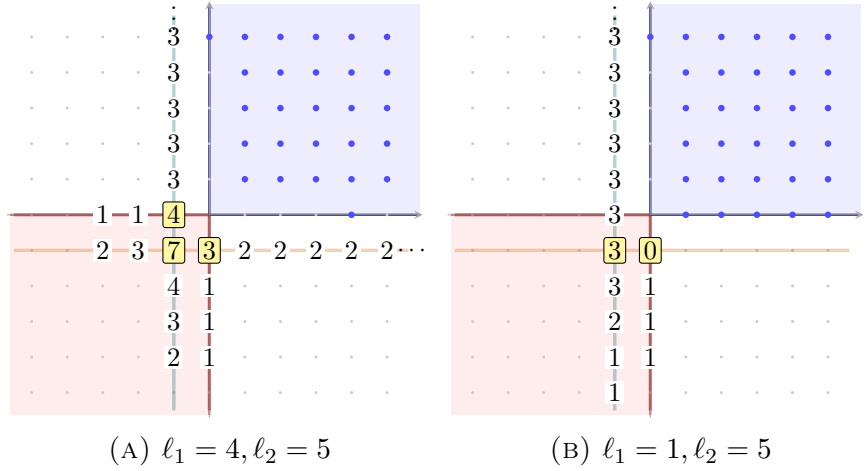
\end{example}

The following direct consequence for normal affine toric surfaces is already known due to \cite{zbMATH00681806}.
\begin{corollary}
\label{cor:surface-t1-normal}
 For the normalisation $\tilde X$ we have  \( \dim (T_{\tilde X}^1)_u =\tilde d(u). \)
\end{corollary}

To prove Theorem~\ref{thm:T1dim-surface} we first determine a generating set for $S$.   The proof of the following Lemma is straightforward and left to the reader.
\begin{lemma}
\label{lem:generating-set}
With the choice of $\bar E=\{w_0,\ldots,w_{s+1}\}$ as above, the set
\begin{align*}
    E= &\{\ell_1 w_0,\ell_2 w_{s+1}\} \cup \{w_2,\ldots,w_{s-1}\} \cup \\
       & \cup\{w_1+i w_0, w_s+j w_{s+1} \mid 0 \leq i \leq \ell_1-1, 0 \leq j \leq \ell_2-1 \}\\
  =&\{u_{1}^{0}(\ell_1),u_{2}^{0}(\ell_2)\} \cup \{w_2,\ldots,w_{s-1}\} \cup \\
        & \cup\{u_{1}^{1}(i),u_{2}^{1}(j) \mid 0 \leq i \leq \ell_1-1, 0 \leq j \leq \ell_2-1 \}
 \end{align*}
  is a generating set for $S$.
\end{lemma}
The generating set from  Lemma~\ref{lem:generating-set} will be denoted by  $E$  for the rest of this section. As in the previous sections we consider the subsets
\(E^u_S=\{w \in E \mid u+w \notin S\}\) and 
\(E^{u,i}:=E^{u}_{S_i}=\{w \in E \mid u+w \notin S_i\}.\)

\begin{lemma}
  \label{lem:surface-T1-from-intersection}
 The dimension of $(T^1_X)_u$ is given by
 \begin{equation}
   \dim((T^1_X)_u)=\max \left\{0, \; 2 + \sum_{i=1,2} (\dim(T^1_{U_i})_u- c_{E^{u,i}}) -\# \bigcap_{i=1,2}E^{u,i} \right\}.\label{eq:dimT1-intersection}       
 \end{equation}
  Here $c_{E^{u,i}}$ is a correction term defined as $c_{E^{u,i}}=\max\{0, 2-\# E^{u,i}\}$.
\end{lemma}

\begin{proof}
  Let us abbreviate the subspace \(\langle q \in \Rel(E^{u,i}) \mid \bar q+u \notin S_i\rangle_k \subset \Rel_k(E^u_S)\) by $L_i$ for the moment. With this notation we have
  \begin{equation}
    \dim (L_1+L_2) =   \dim L_1+ \dim L_2- \dim (L_1 \cap L_2).\label{eq:surface-lemma-dim-formula}
  \end{equation}
  On the other hand, Proposition~\ref{prop:quotients-under-localisation} implies the identity
  \begin{equation}
    \dim L_i = \dim \Rel_k(E^{u,i}) - \dim (T_{U_i}^1)_u = \# E^{u,i}-2+c_{E^{u,i}} - \dim (T_{U_i}^1)_u\label{eq:surface-lemma-Li}.
  \end{equation}
  Plugging this into (\ref{eq:surface-lemma-dim-formula}) gives
   \begin{equation}
    \dim (L_1+L_2) =   \# E^{u,1}+ \# E^{u,2}-4 + c_{E^{u,1}}+ c_{E^{u,2}}-\dim (L_1 \cap L_2).\label{eq:surface-lemma-dimsum}
  \end{equation}  
  On the other hand, by Lemma~\ref{lem:S2-T1} we either have
  \begin{equation}
    \dim (T_{X}^1)_u= \dim \Rel_k(E^u_S)-  \dim(L_1+L_2)=\# E^u_S-2 - \dim(L_1+L_2) \label{eq:surface-lemma-T1}
  \end{equation}
  or $\dim (T_{X}^1)_u=0$ if $\# E^u_S< 2$ and therefore $\dim \Rel_k(E^u_S)=0$ holds.
  Notice, that we have $E^u_S=E^{u,1} \cup E^{u,2}$, leading to the identity
  \begin{equation}
    \# E^u_S-\# E^{u,1}- \# E^{u,2}=-\#(E^{u,1}\cap E^{u,2}).\label{eq:surface-lemma-inclusion-exclusion}
  \end{equation}
Combining (\ref{eq:surface-lemma-dimsum}), \ref{eq:surface-lemma-T1}) and (\ref{eq:surface-lemma-inclusion-exclusion})
gives 
\begin{equation}
  \label{eq:surface-lemma-T1-formula}
  2-\#\bigcap_{i=1,2}(E^{u,i}) + \sum_{i=1,2}\dim (T_{U_i}^1)_u  -\sum_{i=1,2}c_{E^{u,i}} + \dim(L_1\cap L_2).
\end{equation}
for $\dim (T_{X}^1)_u$. 
Now, the claim  would follow if we could show that $\dim (T_{X}^1)_u=0$ holds whenever we have $\dim(L_1\cap L_2) > 0$. To see this, assume first that $(T_{U_1}^1)_u$ and $(T_{U_2}^1)_u$ vanish. Then we have
\[
L_1\cap L_2 = \Rel_k(E^{u,1})\cap \Rel_k(E^{u,2}) = \Rel_k(E^{u,1} \cap E^{u,2}).
\]
The first identity follows from Lemma~\ref{prop:quotients-under-localisation} and the second one directly from the definition. Now, under the assumption $\dim(L_1\cap L_2) > 0$ we have $\dim (L_1 \cap L_2)=\# (E^{u,1} \cap E^{u,2})-2$. Plugging this into (\ref{eq:surface-lemma-T1-formula}) gives indeed $\dim (T_{X}^1)_u=0$.
For the remaining cases see Lemma~\ref{lem:non-trivial-intersection} below.
\end{proof}

\begin{lemma}
  \label{lem:non-trivial-intersection}
  If $(T_{U_1}^1)_u \neq 0$ or $(T_{U_2}^1)_u \neq 0$ and the intersection \[\langle q \in \Rel(E^{u,1}) \mid \bar q+u \notin S_1\rangle_k \;\cap\; \langle q \in \Rel(E^{u,2}) \mid \bar q+u \notin S_2\rangle_k\]
  is non-trivial, then $\dim (T_{X}^1)_u$ vanishes.
\end{lemma}
\begin{proof}
  The intersection  is contained in $\Rel_k(E^{u,1} \cap E^{u,2})$, but $E^{u,1} \cap E^{u,2}$ is empty if $-u \notin \relint\sigma^\vee$. Hence, we may assume
  $-u \in \relint\sigma^\vee$.
  
  We distinguish two cases.
  
  \textbf{Case  $\langle v_1,u \rangle = -1$:}
  In general for $u=-u_{1}^{1}(j)$ we have 
  \[E^{u,1} = \{u_{1}^{0}(1)\} \cup \{ u_{1}^{1}(i) \mid 1 \leq i < \ell_1, i \not\equiv j \;(\text{mod } \ell_1)\}.\]
  In particular, we have $\# E^{u,1}=\ell_1$ and $\dim\Rel_k(E^{u,1})=\ell_1-2$. Since, $\dim (T^1_{U_1})_u=\ell_1-2$
  by Lemma~\ref{lem:t1} it follows via Proposition~\ref{prop:quotients-under-localisation} that \[\langle q \in \Rel(E^{u,1}) \mid \bar q+u \notin S_1\rangle_k\] vanishes. If $\langle v_1, \ell_2 g_2 \rangle =1$ then additionally $g_2$ can be contained in $E^{u,1}$. In this case the same argument shows that $\dim \langle q \in \Rel(E^{u,1}) \mid \bar q+u \notin S_1\rangle_k =1$ and every element of $\langle q \in \Rel(E^{u,1}) \mid \bar q+u \notin S_1\rangle_k$ is necessarily supported in $g_1 \notin E^{u,2}$. Hence, once again the claim follows.
  
  \textbf{Case  $\langle v_1, u \rangle = -2$, $\ell_1=2$ and $u \in 2M$:} We distinguish two subcases. The first one is $u=-2u_{1}^{1}(1)$. Here, $E^{u,1}$ contains $u_{1}^{0}(1)$, $u_{1}^{1}(1)=w_1$, $u_{1}^{1}(2)=w_1+w_0$ and if $a_1=1$ then also $w_2$. In any case $\# E^{u,1} \leq 4$ holds and since $\dim (T_{U_1}^1)_u=1$, we have
  \(
    \dim \langle q \in \Rel(E^{u,1}) \mid \bar q+u \notin S_1\rangle_k = \# E^{u,1}-2-1 \leq 1
  .\)
  As in the previous case it follows that every element of $\langle q \in \Rel(E^{u,1}) \mid \bar q+u \notin S_1\rangle_k$ is supported in $g_1 \notin E^{u,2}$ and, therefore, the claim follows.

  The second subcase is $u=-2u_{1}^{1}(m)$ with $m \geq 2$. Then $\langle v_2, u\rangle < -2$ and therefore
  $(T^1_{U_2})_u$ vanishes. Hence, $\langle q \in \Rel(E^{u,2}) \mid \bar q+u \notin S_2\rangle_k=\Rel_k(E^{u,2})$.
  On the other hand, we have $E^{u,2}=E^u_S$ in this situation. Hence, we conclude $(T^1_X)_u=0$ by Lemma~\ref{lem:S2-T1}.
\end{proof}

\begin{remark}
  \label{rem:correction-terms}
  We may use Lemma~\ref{lem:t1} to determine $\dim(T^1_{U_i})_u$ and count the elements in $E^{u,i}$ to determine 
  $c_{E^{u,i}}$.  We obtain in this way
  \begin{equation}
    \label{eq:correction-terms}
    \dim(T^1_{U_i})_u- c_{E^{u,i}}=
    \begin{cases}
      -1  &\langle u, v_i\rangle = 0, u \notin \ell_iM\\
      -2  &\langle u, v_i\rangle = 0, u \in \ell_iM\\
%      \midrule
      \ell_i-2  &\langle u, v_i\rangle = -1\\
      1 & \langle u, v_i \rangle = -2, u \in 2M, \ell_i=2\\
  %    \midrule
      -2  &\langle u, v_i\rangle > 0\\
      0 & \text{otherwise.}
    \end{cases}
  \end{equation}
\end{remark}

\begin{lemma}
  \label{lem:intersection}
  For $u \in M$ we have
  \begin{equation}
    \label{eq:intersection}
    \# \bigcap_{i=1,2}E^{u,i}  =
    \begin{cases}
      0 & u = -u_{i}^{0}(m), m \leq \ell_i \text{ or } m \in \ell_i\ZZ\\
      1 & u = -u_{i}^{0}(m), m > \ell_i, m \notin \ell_i\ZZ\\ 
      %\midrule
      0 & u = -u_{i}^{1}(m),   m < 1\\
      m-1 & u = -u_{i}^{1}(m), 1\leq  m \leq \ell_i\\
      \ell_i& u = -u_{i}^{1}(m),  m \geq \ell_i+1 \\
      %\midrule
      0 & u=-w_i, 1 \leq i \leq s\\
      1 & u=-mw_i, 1 \leq i \leq s, 2 \leq m \leq a_i-1\\
      0 & u \notin -\sigma^\vee.
    \end{cases}
  \end{equation}
  For all remaining values of $u$ we have $\# \bigcap_{i=1,2} E^{u,i} \geq 2$, with the following exceptions.
  \begin{enumerate}[(a)]
  \item If $s=1$, $u=-a_1w_1, \ell_1,\ell_2 \geq 2$ or  $u=-a_1w_1, \ell_1=\ell_2 = 1$ then 
    $\# \bigcap_{i=1,2}E^{u,i}=1$ holds.
    \label{item:exception-multiple}
  \item If $\tilde X \cong \A^2$, $u=-u_{2}^{1}(\ell_2), \ell_1 \neq  1$ or  $u=-u_{1}^{1}(\ell_1), \ell_2 \neq  1$, then  the number increases by $1$ compared to (\ref{eq:intersection}).
    \label{item:exception-smooth-perp}
  \item If $\tilde X\cong \A^2$, $u=-u_{2}^{1}(m), \ell_1 = 1,m \notin \ell_2\ZZ$ or  $u=-u_{1}^{1}(m), \ell_2 = 1, m \notin \ell_1\ZZ$, then the number increases by $1$ compared to (\ref{eq:intersection})
    \label{item:exception-smooth-parallel}
  \end{enumerate}
\end{lemma}
\begin{proof}
  The statements in the case $u \notin \sigma^\vee$ is clear. For  $u=u_{i}^{0}(m)$ and $m > \ell_i, m\notin \ell_i\ZZ$ the intersection consists exactly of the element $\ell_i g_i$. For $m < \ell_i$ or $m \in \ell_i \ZZ$ the  intersection is empty.
  
  An element $-u \in \relint\sigma^\vee \cap M$ can be expressed as positive integral linear combination $\lambda_i w_j + \lambda_{j+1}w_{j+1}$ of two consecutive elements of the generating set $\bar E$. If both coefficients are non-zero we have $\# \bigcap_{i=1,2}\bar E^{u,i} \geq 2$. Indeed, either $w_j,w_{j+1}$, $w_1,w_0+w_1$ or $w_s,w_s+w_{s+1}$ are contained in the intersection. We now consider multiples of the elements $w_1,\ldots,w_{s}$. For $-u=mw_j$ we have $\bigcap_{i=1,2} E^{u,i}=\emptyset$ for $m=1$ and  $\bigcap_{i=1,2}E^{u,i}=\{w_j\}$ for $2\leq m \leq a_j-1$. If we are in condition in (\ref{item:exception-multiple}), then we also have  $\bigcap_{i=1,2}E^{u,i}=\{w_j\}$. Otherwise for $-u=a_jw_j$ at least one of the elements $w_{j-1}$, $w_{j+1}$, $w_0+w_1$, $w_s+w_{s+1}$ will be also contained in the intersection. For $m > a_j$ we have $\#\bigcap_{i=1,2} E^{u,i} \geq 3$, since $w_{j-1},w_j,w_{j+1}$  (or $w_0+w_1,w_1,w_2$ or $w_{s-1},w_s,w_s+w_{s+1}$, respectively) are contained in the intersection.

  We now consider the case $u=u_{1}^{1}(m)$.  The intersection
  \[E^{u,1} \cap E^{u,2}
    = \begin{cases}
        \{u_{1}^{1}(n) \mid 1 \leq n  \leq  m\} & m \leq \ell_1\\
        \{u_{1}^{1}(n) \mid n  < \ell_1, n \not\equiv m \; (\text{mod } \ell_1)\} \cup \{u_{1}^{0}(\ell_1)\} & m \geq \ell_1
    \end{cases}
  \]
  has the claimed number of elements.  In the exceptional cases (\ref{item:exception-smooth-perp}) and (\ref{item:exception-smooth-parallel}) the element  $u_{1}^{0}(\ell_1)$ or $u_{2}^{0}(\ell_2)$, respectively,  is additionally contained in the intersection, since we have $u_{1}^{1}(m)=u_{2}^{0}(1) +m\cdot u_{1}^{0}(1)$ in case that $\tilde X\cong \A^2$. Hence, the cardinality of the intersection increases by $1$.  The argument for $u=-u_{2}^{1}(m)$ is  the same.
\end{proof}

  \begin{proof}[Proof of Theorem~\ref{thm:T1dim-surface}]
    According to Lemma~\ref{lem:surface-T1-from-intersection} the dimension of $(T^1_X)_u$ is given by (\ref{eq:dimT1-intersection}).   The values for the correction terms $\dim(T^1_{U_i})_u- c_{E^{u,i}}$ for $i=1,2$ are listed in (\ref{eq:correction-terms}) in Remark~\ref{rem:correction-terms}  and apart from the four exceptional cases the cardinalities $\bigcap_{i=1,2}E^{u,i}$ can be found in (\ref{eq:intersection}) in Lemma~\ref{lem:intersection}. It is tedious, but straightforward to check that combining both results in (\ref{eq:dimT1-surfaces}). Let us just comment on the cases $u=-u_{1}^{1}(m)$ with $2 \leq m  \leq \ell_1$. Here, the expression in (\ref{eq:dimT1-intersection}) becomes
         \[2+ (\ell_1 -2) +0 - m+1=(\ell_1-m+1)+0+0=d_1(u)+d_2(u)+\tilde d(u).\]
For $m=1$, but $s > 1$ (in this case we have $\langle u_{1}^{1}(1), v_2 \rangle < -1$) we get 
         \[2+ (\ell_1 -2) +0 -0=(\ell_1-1)+0+1=d_1(u)+d_2(u)+\tilde d(u).\]
For $m=1$ and $s=1$ we have $u_{1}^{1}(1)=u_{2}^{1}(1)$ and  the expression in (\ref{eq:dimT1-intersection}) gives
\[2+ (\ell_1 -2) +(\ell-2) -0=(\ell_1-1)+(\ell_2-1)+0=d_1(u)+d_2(u)+\tilde d(u).\]
We leave the verification of the remaining cases to the reader.  Note, that the exceptional cases \ref{item:exception-multiple})-(\ref{item:exception-smooth-parallel}) in Lemma~\ref{lem:intersection} directly translate to the Rules~\ref{item:singular-exception-multiple})-(\ref{item:smooth-exception-li=1}) in Theorem~\ref{thm:T1dim-surface}. Let us only comment on Rule~(\ref{item:smooth-exception-l1=l2=2}). For $u=-2w_1=-2w_s$
combining (\ref{eq:dimT1-intersection}) and (\ref{eq:correction-terms}) gives
\begin{equation}
     \dim((T^1_X)_u)=\max\left\{0,\; 2+ \#\{i \mid \ell_i=2\}-\# \bigcap_{i=1,2}E^{u,i}\right\}.\label{eq:dimT1-2w1}
   \end{equation}
   For $\ell_2, \ell_2\geq 2$ we have  \( \bigcap_{i=1,2}E^{u,i}=\{w_1+w_0,w_1,w_1+w_2\}\). Hence, according
   to (\ref{eq:dimT1-2w1}) we get $\dim((T^1_X)_u)=0$ except for $\ell_2=\ell_2=2$. In this case we get $\dim((T^1_X)_u)=2+2-3=1$.
   We have   \(\bigcap_{i=1,2}E^{u,i}=\{w_1,w_2\}\)  for $\ell_1=1, \ell_2=2$ and  \(\bigcap_{i=1,2}E^{u,i}=\{w_0,w_1\}\) for $\ell_1=2, \ell_2=1$.
   In both cases we obtain $\dim((T^1_X)_u)=2+1-2=1$  according to (\ref{eq:dimT1-2w1}).
   It remains to consider the case $\ell_1,\ell_2 \neq 2$, but then (\ref{eq:dimT1-2w1}) gives $\dim((T^1_X)_u)=0$, since we have \(\#\bigcap_{i=1,2}E^{u,i} \geq 2\) by Lemma~\ref{lem:intersection}.
 \end{proof}

\printbibliography
\end{document}